\documentclass[11pt]{article}
\usepackage{amssymb}
\usepackage{amsthm}
\usepackage{amsmath,amsfonts,amssymb}

\usepackage[mathscr]{eucal}
\usepackage{exscale}
\usepackage{natbib}
\usepackage{bm}
\usepackage{eqlist}
\usepackage[dvipsnames]{color}
\usepackage[Lenny]{fncychap}
\usepackage{multirow}
\usepackage{graphicx}
\usepackage{algpseudocode}
\usepackage{algorithm}
\usepackage{caption}
\usepackage{hyperref}
\usepackage{framed}
\usepackage{color}
\usepackage{booktabs}
\usepackage{longtable}

\usepackage{framed}
\usepackage{algpseudocode}

\usepackage{tabu}
\usepackage{tikz}
\usepackage[title]{appendix}

\usepackage[utf8]{inputenc} 
\usepackage[english]{babel} 
\usepackage[T1]{fontenc}    
\usepackage{amssymb,amsmath,amsfonts} 
\usepackage{braket} 
\usepackage{caption}
\usepackage{algorithm}

\hypersetup{
    colorlinks=true,
    linkcolor=blue,
    citecolor=blue,
    citebordercolor={1 1 0},
}

\renewcommand{\theequation}{\thesection.\arabic{equation}}

\newtheorem{theorem}{Theorem}[section]

\newtheorem{proposition}[theorem]{Proposition}

\newtheorem{definition}[theorem]{Definition}

\font\bigbf=cmbx10 scaled \magstep3

\begin{document}

\title{\bigbf  Fairness in online vehicle-cargo matching: An intuitionistic fuzzy set theory and tripartite evolutionary game approach}

 \author{Binzhou Yang
 \quad Ke Han$\thanks{Corresponding author, e-mail: kehan@swjtu.edu.cn;}$
 \quad Wenrui Tu
 \quad Qian Ge
 \\\\
 \textit{\small Institute of System Science and Engineering, School of Transportation and Logistics,}\\
 \textit{\small Southwest Jiaotong University}
 }

\maketitle

\begin{abstract}

This paper explores the concept of fairness and equitable matching in an on-line vehicle-cargo matching setting, addressing the varying degrees of satisfaction experienced by shippers and carriers. Relevant indicators for shippers and carriers in the on-line matching process are categorized as attributes, expectations, and reliability, which are subsequent quantified to form satisfaction indicators. Employing the intuitionistic fuzzy set theory, we devise a transformed vehicle-cargo matching optimization model by combining the fuzzy set's membership, non-membership, and uncertainty information. Through an adaptive interactive algorithm, the matching scheme with fairness concerns is solved using CPLEX. The effectiveness of the proposed matching mechanism in securing high levels of satisfaction is established by comparison with three benchmark methods. To further investigate the impact of considering fairness in vehicle-cargo matching, a shipper-carrier-platform tripartite evolutionary game framework is developed under the waiting response time cost (WRTC) sharing mechanism. Simulation results show that with fairness concerns in vehicle-cargo matching, all stakeholders are better off: The platform achieves positive revenue growth, and shippers and carriers receive positive subsidy. This study offers both theoretical insights and practical guidance for the long-term and stable operation of the on-line freight stowage industry.
\end{abstract} 


\section{Introduction}

In recent years, the swift advancement of logistics information integration platforms has facilitated the digitization and informatization of logistics activities, promoting real-time sharing of supply and demand information in the sector. This progress in online freight stowage has enhanced information transparency and, to some extent, reduced freight costs compared to traditional freight transportation forms. What followed are various researches on online freight stowage, mainly focusing on operation modes \citep{HAWJGSW2015,YLXDBL2017}, vehicle-cargo matching decision-making \citep{LZDY2020,WLGGW2020,TWMLG2022} and path planning \citep{HQRD2015,BI2019,TAJD2020,MDTA2023,SGH2023,RC2023}.

Vehicle-cargo matching problems are central to online freight stowage platforms. As a quintessential example of bilateral matching, their primary objective is to facilitate transactions between shippers and carriers. Existing studies highlight that the efficiency of road transportation, based on its resources benefits from the effective operation of logistics information platforms \citep{AAS2003}. Indeed, service providers should strive to enhance their innovation capabilities, enabling them to compete more effectively \citep{B2010}. Consequently, it is advantageous for online platforms to integrate freight capacity resources and augment the comprehensive efficiency of freight stowage systems by pursuing innovative vehicle-cargo matching mechanisms.

Fairness concerns have gained significant attention in the fields of logistics and supply chain management. It is important to note that the concept of fairness is largely contingent on a stakeholder's subjective perception within a specific decision-making context. Behavioral economics research demonstrates that perceptions of unfairness influence individuals' willingness to engage in certain behaviors \citep{KKT1986,HB2013}. Moreover, significant disparities in participants' payoffs under a given outcome can lead to the rejection of that solution \citep{GSS1982}. In fact, when the vehicle-cargo matching scheme is inclined towards unilateral matching subjects, it shatters the online freight transaction interactions, impeding the optimal allocation of freight resources within the online freight industry. This can result in protests and strikes, ultimately terminating matching transactions between carriers and shippers \citep{SoHu2023, Zhihu2023}. In this study, referring to classical criterion \citep{KM2015, HSZ2016}, we view fairness as the unequal status of parties within a match, leading to differing perceptions (satisfaction) of the same matching outcome. However, to our knowledge, no previous studies have considered fairness in an online vehicle-cargo matching context. This paper explores the concept of fairness in an online vehicle-cargo matching setting, addressing the varying degrees of satisfaction experienced by shippers and carriers. In real-life vehicle-cargo matching scenarios, the satisfaction degree of unilateral matching subjects is often uncertain or ambiguous. To more accurately capture the nebulous nature of bilateral matching decisions, intuitionistic fuzzy set theory has been extensively investigated in various contexts \citep{JWZZZ2021,LYL2021}. Motivated by this, our study focuses on modeling the unequal matching status perception in conjunction with intuitionistic fuzzy set theory (see Section \ref{3.0} for more details).

The subsequent impact and effectiveness test of fairness concern in practical application is tightly associated with stakeholders' decision-making environments. In our study, the online vehicle-cargo matching platform serves as the central facilitator of two-sided transactions between carriers and shippers. The platform plays a critical role in developing vehicle-cargo matching schemes that consider fairness. We then examine stakeholders' strategic decision-making responses to the implemented vehicle-cargo matching scheme, based on cost-sharing contracts. Evidently, the platform possesses greater authority in the vehicle-cargo matching scheme formulation process, with carriers and shippers as recipients. As carriers and shippers make strategic selections in response to the platform's matching scheme, the impact of the scheme's development and implementation on the strategic selection cannot be overlooked.

\begin{figure}[!htbp]
\centering
\includegraphics[width=.98\textwidth]{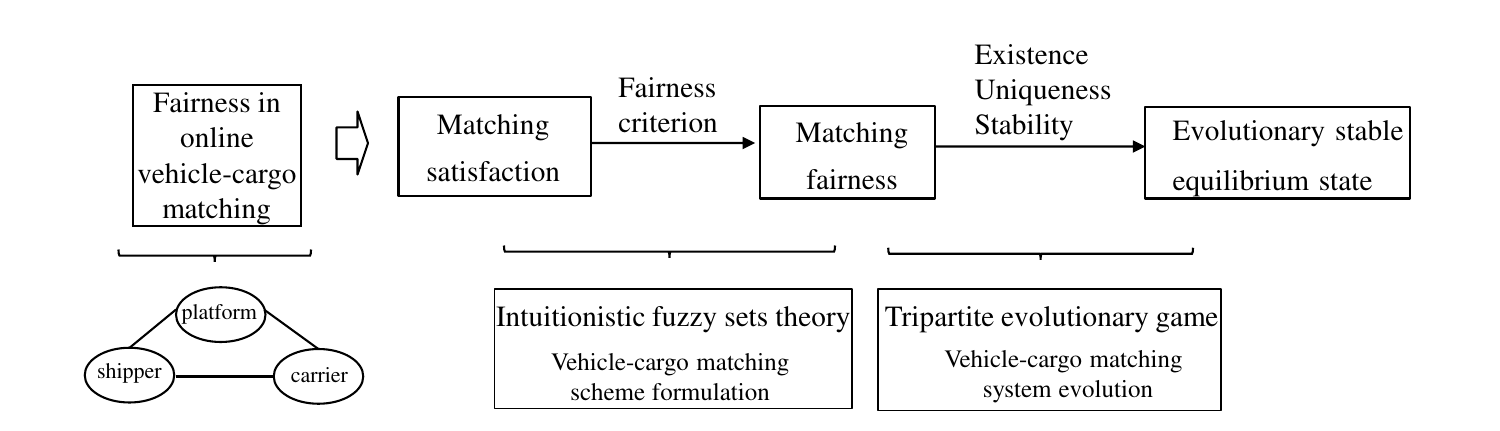} 
\caption{Framework of fairness in online vehicle-cargo matching}
\label{Framework}
\end{figure}

In essence, this paper seeks to address the existing gap in considering fairness during the formulation and implementation of vehicle-cargo matching schemes by developing an analytical model and exploring its practical implications (see Figure \ref{Framework}). On one hand, an optimization model is created to maximize the overall satisfaction of bilateral vehicle-cargo matching entities, factoring in the satisfaction realization degree of unilateral matching subjects. The traditional vehicle-cargo matching model is transformed using intuitionistic fuzzy set theory, and methods for determining satisfactory matching schemes (explicitly trade the fairness and satisfaction objective) are discussed. On the other hand, a tripartite evolutionary game framework is proposed to examine the evolutionary trajectory of the online vehicle-cargo matching system. Stakeholders' strategic decision-making behaviors are analyzed concerning various influencing factors, particularly fairness concerns, and a cost-sharing contract is incorporated into the tripartite evolutionary game model to more realistically expand the research direction. Our primary findings and contributions are as follows.

1.	The satisfactory vehicle-cargo matching scheme, derived using intuitionistic fuzzy set theory, aligns with the decision-maker's understanding and depiction of realistic objectives. The matching outcome is closely related to 1) the value interval of the fairness concern coefficient, and 2) the limit of unilateral matching subjects' unsatisfactory interval restrictions. The proposed matching mechanism's advantages are highlighted when compared to traditional bi-objective optimization approaches, such as max-min, ideal point solutions, and linear weighting.

2.	For the entire online vehicle-cargo matching tripartite evolutionary system, the system will evolve towards the anticipated evolutionary stable equilibrium state when platform subsidy intensity, service level, and fairness concern intensity satisfy specific conditions. The evolution rate and direction depend on the influencing factors' relative strength.

3.	In cases where carriers and shippers have inconsistent strategic selections, the matching subjects who decline the recommended matching scheme should bear all waiting response time costs (WRTC). This approach facilitates the tripartite evolutionary system's achievement of the expected evolutionary stable equilibrium state through a cost-sharing contract mechanism.

Contributions:

•	This paper models the membership, non-membership, and uncertainty information of unilateral matching subjects' satisfaction realization degrees based on intuitionistic fuzzy set theory. This approach allows for flexible degrees of alignment and trade-offs between two optimization objectives in line with decision-maker's logic. Consequently, a satisfactory vehicle-cargo matching scheme that considers fairness concerns is obtained.

•	Unlike existing literature that typically considers only two parties in online matching, this paper proposes a tripartite evolutionary game model that includes shippers and carriers participating in vehicle-cargo matching, and a platform making matching and subsidizing decisions. The existence, uniqueness, and stability of the equilibrium state are discussed in relation to different model parameter choices.

•	This paper is among a few pieces of literature that explore the unequal perceptions of both matching sides on the same outcome in the online vehicle-cargo matching process, addressing the. With the proposed fairness-aware vehicle-cargo matching scheme, the platform achieves positive revenue growth while shippers and carriers receive subsidy income when the system reaches equilibrium. In contrast, they do not receive such income without fairness considerations (see Section \ref{5.0} for detailed results and analysis).

The remainder of this paper is organized as follows: Section \ref{2.0} presents the literature review. Section \ref{3.0} focuses on the vehicle-cargo matching problem formulation, improvement, and interactive solving based on intuitionistic fuzzy set theory. Section \ref{4.0} analyzes stakeholders' strategic selection with fairness concerns. Section \ref{5.0} examines a case study using typical online vehicle-cargo matching scenarios and provides the corresponding simulation analysis. Finally, major conclusions, managerial insights, and future research directions are outlined.

\section{Literature review} \label{2.0}
In this section, we provide a research survey of fairness concerns, vehicle-cargo matching problems and intuitionistic fuzzy set theory in bilateral matching. Our study will fill the niche of fairness in online vehicle-cargo matching. 
\subsection{Fairness concerns in logistics systems and supply chains}

When discussing fairness concerns in logistics supply chains, most existing research focuses on pricing decisions, primarily expanding from three aspects: the reference point of fairness concern in specific problem settings, stakeholders' decision-making behavior environments, and related contract mechanisms.

Since profit distribution among stakeholders in a supply chain introduces fairness concerns in utility realization, stakeholders' payoffs in specific problem settings are often modeled as the reference point \citep{LP2018,SGH2023}. Scholars have also explored the impact of coexisting different fairness concern reference types on stakeholders' optimal decisions, such as distributional fairness concerns and dual/peer-induced fairness concerns \citep{HSW2014,LWSYW2018,DWZN2018}. Fairness concern alters stakeholders' relationships, which is typically incorporated into various game-theoretic models concerning stakeholders' decision-making behavior environments, including Stackelberg games \citep{AB2020,WCW2022}, ultimatum games \citep{HSW2014}, and others \citep{JZZC2023}. Simultaneously, to enhance logistics production operation channel coordination capabilities, different contract mechanisms are introduced into game-theoretic models, ensuring decision-makers' utilities can be maximized when specific parameter conditions are met \citep{ND2017,LWPS2022,DHY2023}. In short, fairness concern is a critical research issue in the behavioral operations management of supply chains.   

Despite the importance of fairness concerns in logistics systems, it has received minimal attention in vehicle-cargo matching. As a major research area for online freight stowage, this paper introduces matching fairness concerns to improve the matching mechanism for such a process.

\subsection{Vehicle-cargo matching problems}

In order to fully utilize online freight stowage information and enhance freight stowage efficiency, scholars have developed models incorporating factors such as matching satisfaction, matching number, matching cost, and matching probability. They have solved these problems using exact algorithms \citep{LZDY2020}, problem-specific heuristic algorithms \citep{YG2013,ZLS2018}, and machine-learning algorithms \citep{DZW2021,TWMLG2022} to obtain vehicle-cargo matching scheme solutions.

As concerns about service increase, the applicability of vehicle-cargo matching mechanisms has gained more attention, focusing on stakeholders' decision-making behavior during the game process. \cite{JHD2017} constructed a bargaining game model for two-sided users based on two-sided market theory and evolutionary game theory. They discussed the control process of evolution from multi-homing to single-homing users aiming to raise the single-homing user ratio of vehicle-cargo matching platforms and enhance matching efficiency and the value of platform enterprise. Based on shippers' and carriers' basic strategic selection of “acceptance” or “non-acceptance” to the vehicle-cargo matching schemes, \cite{WLGGW2020} investigated several factors' influence on the evolutionary stable equilibrium state of vehicle-cargo matching systems, such as online platform's service level and users' waiting cost.  It is worth noting that the online vehicle-cargo matching process exhibits two-sided market characteristics: users' decision behavior is interdependent and mutually influential  \citep{A2006,R2009}.

Little attention, however, has been given to the impact of matching fairness concerns on the vehicle-cargo matching system's evolution process, and the platform's role in this process has been overlooked. This paper explores how the consideration of fairness by the platform when formulating and implementing vehicle-cargo matching schemes impacts stakeholders' subsequent strategic selection. As a result, the platform is incorporated into the tripartite evolutionary game framework in our vehicle-cargo matching mechanism research, distinguishing it from existing literature.

\subsection{Intuitionistic fuzzy set theory in bilateral matching }
In many decision scenarios, it is difficult to describe subjective, vague, and imprecise information. Therefore, fuzzy set theory \citep{Z1965} has been introduced in various bilateral matching decision problem studies. \cite{A1986,A1999} added the concepts of non-membership and uncertainty to the basis of fuzzy set theory, proposing an intuitionistic fuzzy set theory. For a comprehensive review of its evolutionary process, the reader can refer to \citep{YSX2022}. Compared with fuzzy sets, intuitionistic fuzzy sets offer both flexibility and practicability, better describing fuzzy information in detail. As a result, intuitionistic fuzzy sets have been widely applied in bilateral matching decision problems, such as public-private partnership projects \citep{WSZ2016}, personnel-position matching problem \citep{LYL2021}, matching problem of the supply and demand of technology/knowledge \citep{JWZZZ2021}, combined with decision behavioral factors, such as prospect theory \citep{TXGH2018}, regret theory \citep{JWZZZ2021}. Although the application of intuitionistic fuzzy expands the research context of the bilateral matching decision, matching stability \citep{ZKPYG2019} and matching satisfaction 
 \citep{LYL2021} remain the primary modeling footholds, in the absence of matching fairness concerns.

Acknowledging the inherent fuzzy characteristics in the decision-making process, this paper novelly incorporates fairness concerns into the vehicle-cargo bilateral matching model using intuitionistic fuzzy set theory. This is achieved by representing the membership, non-membership, and uncertainty information of unilateral matching subjects' satisfaction realization degrees.

For reference and comparison, we present a summary of related literature in logistics and supply chain management that considers fairness, as shown in Table \ref{relatedreserachtab}:

\begin{table}[!h]
\caption{A summary of related logistic supply chain researches considering fairness concern}	\label{relatedreserachtab} 
\resizebox{\textwidth}{1.45in}{
\begin{tabular}{|c|c|c|c|c|}
\hline
Literature                                                         & \begin{tabular}[c]{@{}c@{}}Problem \\ setting\end{tabular}                      & \begin{tabular}[c]{@{}c@{}}Reference of \\ fairness concerns\end{tabular}       & \multicolumn{1}{l|}{Type of game}                            & \begin{tabular}[c]{@{}c@{}}Contract \\ mechanism\end{tabular}             \\ \hline
\begin{tabular}[c]{@{}c@{}} \cite{HSW2014} \end{tabular}         & \begin{tabular}[c]{@{}c@{}}a dyadic \\ supply chain\end{tabular}                & \begin{tabular}[c]{@{}c@{}}utility \\ realization\end{tabular}                  & \begin{tabular}[c]{@{}c@{}}ultimatum \\ games\end{tabular}   & ——                                                                        \\ \hline
\begin{tabular}[c]{@{}c@{}} \cite{ND2017}\end{tabular}        & \begin{tabular}[c]{@{}c@{}}a dyadic \\ supply chain\end{tabular}                & \begin{tabular}[c]{@{}c@{}}utility \\ realization\end{tabular}                  & \begin{tabular}[c]{@{}c@{}}stackelberg \\ game\end{tabular}  & \begin{tabular}[c]{@{}c@{}}quantity discount;\\  fixed fees\end{tabular}  \\ \hline
\begin{tabular}[c]{@{}c@{}} \cite{LWSYW2018} \end{tabular}        & order allocation                                                                & \begin{tabular}[c]{@{}c@{}}utility \\ realization\end{tabular}                  & \begin{tabular}[c]{@{}c@{}}stackelberg \\ game\end{tabular}  & \begin{tabular}[c]{@{}c@{}}membership fee; \\ profit sharing\end{tabular} \\ \hline
\begin{tabular}[c]{@{}c@{}}\cite{AB2020}\end{tabular} & \begin{tabular}[c]{@{}c@{}}production of green \\ apparel products\end{tabular} & \begin{tabular}[c]{@{}c@{}}utiltiy \\ realization\end{tabular}                  & \begin{tabular}[c]{@{}c@{}}stackelberg \\ game\end{tabular}  & \begin{tabular}[c]{@{}c@{}}cost-sharing;\\ profit-sharing\end{tabular}    \\ \hline
\begin{tabular}[c]{@{}c@{}} \cite{WCW2022} \end{tabular}   & \begin{tabular}[c]{@{}c@{}} closed-loop \\supply chain \end{tabular}         & \begin{tabular}[c]{@{}c@{}}utility \\ realization\end{tabular}                  & \begin{tabular}[c]{@{}c@{}}stackelberg \\ game\end{tabular}  &  ——                                                                        \\ \hline
\begin{tabular}[c]{@{}c@{}} \cite{JZZC2023} \end{tabular}                                                           & \begin{tabular}[c]{@{}c@{}}multi-agent reverse\\supply chain\end{tabular}               & \begin{tabular}[c]{@{}c@{}} utility \\ realization \end{tabular} & \begin{tabular}[c]{@{}c@{}}biform \\ game\end{tabular} & coordination                                                             \\ \hline
This paper                                                         & \begin{tabular}[c]{@{}c@{}}vehicle-cargo \\ matching\end{tabular}               & \begin{tabular}[c]{@{}c@{}}matching \\ satisfaction \\ realization\end{tabular} & \begin{tabular}[c]{@{}c@{}}evolutionary \\ game\end{tabular} & cost-sharing                                                              \\ \hline
\end{tabular}}
\end{table}

\section{Vehicle-cargo matching model formulation, improvement and solving based on intuitionistic fuzzy set theory with fairness concern} \label{3.0}
In this section, based on intuitionistic fuzzy set theory, we will conduct our work of vehicle-cargo matching problem modelling, improvement and solving with fairness concern. 
\subsection{Matching fairness criteria}\label{3.1}
  Existing vehicle-cargo matching reseraches mainly focus on maximizing matching satisfaction aiming at fully utilizing freight resources and improving freight efficiency in the absence of fairness concern. It is worth noting that fairness concern brings fresh service innovation by influencing stakeholders' strategic behavior directly. Nevertheless, there is no straightforward way to measure fairness in various decision scenarios. Below, we show some classic fairness criteria that will be used as a reference for measuring matching fairness in the remainder of the paper. The readers can refer to \citep{KM2015,CH2023} for surveys.
  
\begin{itemize}
\item Rawlsian maximin (Max-min fairness).
Max-min fairness (MMF) is a generalization of the
Rawlsian justice solution in the two-player problem. The 
solution corresponds to maximizing the
minimum utility the players derive simultaneously. In other words,
the players simultaneously derive the largest 
possible equal fraction of their respective maximum achievable utilities.
\item Proportionality (Nash bargaining solution).
Proportional fairness (PF) is the generalization of the Nash
solution for a two-player problem. Under the Nash standard,
a transfer of resources between two players is favorable and
fair when the percentage increase in the utility of one player
is larger than the percentage decrease in the utility of the other
player. Proportional fairness is the generalized Nash solution 
for multiple players.
\item Equitability (Alpha fairness).
Alpha fairness regulates the combination of a continuous parameter 
$\alpha$, where larger values of $\alpha$ signify a greater emphasis 
on fairness. A famous special case is the Nash bargaining solution,
which corresponds to $\alpha=1$. Alpha fairness allots the parties the 
largest possible fraction of their potential utility while observing 
fairness by equalizing that fraction across parties.
\end{itemize}

  Fairness concern focuses on ``equitability'' with regard to utility realization in the logistics supply chain area \citep{ND2017,AB2020,WCW2022}. Combined with existing studies about the fairness criterion which has been widely discussed and applied, we adopt the equitability fairness criterion to normalize the vehicle-cargo matching fairness in this paper.

\begin{definition} Vehicle-cargo matching fairness

 We consider a bilateral online freight stowage market composed of shippers, carriers and the online platform. The platform makes matching decisions based on participants' vehicle-cargo matching invitation information. In this paper, the vehicle-cargo matching fairness concern coefficient $\eta$ tightly relates to the matching satisfaction realization from both the shipper side and carrier side, which equals the shipper side's matching satisfaction to the carrier side's matching satisfaction and vice versa.
\end{definition}  

\subsection{Matching satisfaction evaluation indicator system} \label{3.2} 

Through China Wutong Network (www.chinawutong.com), we collected nearly 1,500 pieces vehicle-cargo matching invitation information from both the shipper and carrier sides. Then, combined with word segmentation, semantic merge, clustering, and other methods, the vehicle-cargo matching indicators are extracted conforming to the actual situation \citep{L2003,MU2014}. The indicator of shipper side is formulated as $A=\left \{ A^{1} , A^{2} \cdots  A^{9}  \right \}$ , $A^{f}$ means the $f$th indicator of shipper. The indicator of carrier side is formulated as $B=\left \{ B^{1} , B^{2} \cdots  B^{9}  \right \}$, with $B^{f}$ being the $f$th indicator of carrier. It is important to note that in this paper, the matching subjects' operational qualifications and reputations are condensed into a reliability indicator, as illustrated in Table \ref{indicatortab}. The reliability indicator contributes to satisfaction aggregation, while the remaining non-reliability indicators are divided into two categories: attribute indicators and expectation indicators. Notably, attribute indicators do not participate in the satisfaction aggregation of matching subjects, whereas expectation indicators do.
\begin{table}[!htbp]
\caption{The indicator of shipper and carrier}	\label{indicatortab} 
\resizebox{\textwidth}{.9in}{
\begin{tabular}{|c|c|c|c|c|c|}
\hline
\multirow{10}{*}{Shipper} & Indicator          & Type        & \multirow{10}{*}{Carrier} & Indicator          & Type        \\ \cline{2-3} \cline{5-6} 
                          & The delivery date $A^{1}$ & Expectation &                    & The delivery date$B^{1}$  & Attribute   \\ \cline{2-3} \cline{5-6} 
                          & The delivery price $A^{2}$ & Expectation &                           & The delivery price $B^{2}$ & Expectation \\ \cline{2-3} \cline{5-6} 
                          & The vehicle type $A^{3}$  & Expectation &                           & The vehicle type $B^{3}$  & Attribute   \\ \cline{2-3} \cline{5-6} 
                          & Cargo weight  $A^{4}$  & Attribute   &                           & Vehicle deadweight $B^{4}$ & Attribute   \\ \cline{2-3} \cline{5-6} 
                          & Cargo length  $A^{5}$      & Attribute   &                           & Vehicle length  $B^{5}$   & Attribute   \\ \cline{2-3} \cline{5-6} 
                          & Cargo size    $A^{6}$        & Attribute   &                           & Vehicle volume  $B^{6}$   & Attribute   \\ \cline{2-3} \cline{5-6} 
                          & Place of departure $A^{7}$  & Attribute   &                           & Place of departure $B^{7}$  & Attribute   \\ \cline{2-3} \cline{5-6} 
                          & Place of delivery  $A^{8}$ & Attribute   &                           & Place of delivery $B^{8}$ & Expectation \\ \cline{2-3} \cline{5-6} 
                          & Reliability  $A^{9}$      & —           &                    & Reliability  $B^{9}$     & —           \\ \hline
\end{tabular}}
\end{table}


We now shift our focus to the quantification of indicator satisfaction. As our research is centered on matching fairness concerns, this paper mainly examines the freight market situation where demand and supply are relatively balanced, excluding the influence of unbalanced market demand and supply on the quantification of indicator satisfaction. For expectation indicators, formulas are provided to quantify their satisfaction levels. Additionally, expectation indicators are classified into two categories (i.e., tolerable or intolerable) concerning the matching subjects' matching intentions.

	\setlength\LTleft{0pt}
	\setlength\LTright{0pt}
	\begin{longtable}{@{\extracolsep{\fill}}cl}
		\caption{Mathematical symbols}	\label{tabmath} \\
			\hline
			\multicolumn{1}{c}{\begin{tabular}[c]{@{}c@{}}Constants \\or\\ parameter\end{tabular}} & \multicolumn{1}{c}{\begin{tabular}[c]{@{}c@{}}Explanation\end{tabular}}  \\
			\hline
			$ i $    & Set of shippers,  $i=1,2 \dots m  $ 
			\\
			$j$      & Set of carriers,  $j=1,2\dots n $
			\\
			$f$      & Set of indicators $f=1,2\dots 9 $
			\\
			$\varepsilon_{ij}^{\text{f}}$ &  The indicator satisfaction of $R^i$ towards $S^j$ under $A^f$ 
			\\
			$\alpha_{i}^{f}$ & The indicator value of $R^i$ under $A^f$ 
			\\
			$\alpha_{\max}^{f}$  & The maximum indicator value of $R^i$ under $A^f$, $ \alpha_{\max}^f=\max \left \{ \alpha_{1}^{f},\alpha_{1}^{f},\dots,\alpha_{m}^{f}   \right \} $
			\\
            $\alpha_{\min}^{f}$  & The minimum indicator value of $R^i$ under $A^f$, $ \alpha_{\min}^f=\min \left \{ \alpha_{1}^{f},\alpha_{1}^{f},\dots,\alpha_{m}^{f}   \right \} $
			\\
            $\theta_{ij}^{f}$  & The indicator satisfaction of $S^j$ towards $R^i$ under $B^f$
			\\
            $b_{j}^f$ & The indicator value of $S^j$ under $B^f$
            \\
            $b_{\max}^f$ & The maximum indicator value of $S^j$ under $B^f$ , $ b_{\max}^f=\max \left \{ b_{1}^{f},b_{1}^{f},\dots,b_{n}^{f}   \right \} $
			\\
            $b_{\min}^f$ & The minimum indicator value of $S^j$ under $B^f$ , $ b_{\min}^f=\min \left \{ b_{1}^{f},b_{1}^{f},\dots,b_{n}^{f}   \right \} $
			\\
			$[a_{i}^{fL},a_{i}^{fU}]$     & Shipper $i$'s measurement interval of the delivery date 
	        \\
			 $M$     & A large enough positive number
			\\
			$\theta$ & A parameter quantifying the delivery date indicator satisfaction,$0<\theta\le 1$ 
			\\
			$\omega$       & A parameter quantifying the delivery price indicator satisfaction,$0<\omega\le 1$  
			\\
			$\varphi _{i}$       & A parameter quantifying the vehicle type indicator satisfaction,$0<\varphi_{i}\le 1$  
			\\
			$\tau _{j}$  & A parameter quantifying the place of delivery indicator satisfaction,$0<\tau_{j}\le 1$ 
			\\
			{\begin{tabular}[l]{@{}l@{}} ($b_{j1}^f,b_{j2}^f$\\ ,\dots,$b_{jq}^f$) \end{tabular}} & Carrier $j$'s preference-sequence set under $B^f$
			\\
			$r_{ij}$ & The rank of shipper $i$'s indicator value in carrier $j$'s preference-sequence set \\
            \hline
	\end{longtable} 

For shipper, the delivery date $A^1$ is an interval-type expectation indicator (see Eq. (\ref{shipper-1}) and Figure 
 \ref{shipperdelivery}). The nearer the carrier's delivery date close to the interval left end, the more satisfied the shipper is. Its indicator satisfaction can be formulated as:

\begin{equation}\label{shipper-1}
\varepsilon _{ij}^f=
\begin{cases}
\theta_{i}+ \left ( 1-\theta_{i} \right ) \frac{a_{i}^{fU}-b_{j}^f}{a_{i}^{fU}-a_{i}^{fL}} \qquad b_{j}^f\in \left [ a_{i}^{fL},a_{i}^{fU} \right ] 
\\
-\theta_{i} \frac{a_{i}^{fL}-b_{j}^f}{a_{i}^{fU}-b_{j}^f} \qquad\qquad\qquad$\;$ b_{j}^f < a_{i}^{fL}$\;$ \text{and} $\;$ A^f $\;$ \text{is} $\;$ \text{tolerable}
\\
-\theta_{i} \frac{b_{j}^{f}-a_{i}^{fU}}{b_{j}^{f}-a_{i}^{fL}} \qquad\qquad\qquad$\;$ b_{j}^f > a_{i}^{fU}$\;$ \text{and} $\;$ A^f $\;$ \text{is} $\;$ \text{tolerable}
\\
-M \qquad\qquad\qquad\qquad$\;$$\;$$\;$ b_{j}^f < a_{i}^{fL}$\;$\text{or}$\;$ b_{j}^f < a_{i}^{fU} $\;$\text{and} $\;$A^f $\;$\text{is} $\;$\text{tolerable}
\end{cases}
\end{equation}

\begin{figure}[!htb]
\centering
\includegraphics[width=.5\textwidth]{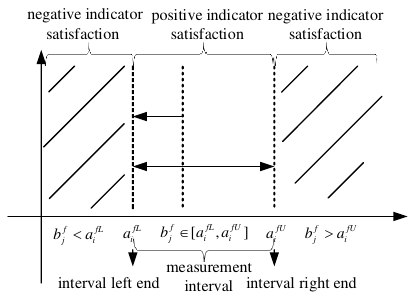}
\caption{Shipper's indicator satisfaction quantification of the delivery date}
\label{shipperdelivery}
\end{figure}

For shipper, the delivery price $A^2$ is a cost-based expectation indicator (see Eq. (\ref{shipper-2}) and Figure \ref{shipperdeliverypricevehicletype}a). The lower the carrier's delivery price, the more satisfied the shipper. Its indicator satisfaction can be formulated as:
\begin{equation}\label{shipper-2}
\varepsilon _{ij}^f=
\begin{cases}
\omega_{i}+ \left ( 1-\omega_{i} \right ) \frac{a_{i}^{f}-b_{j}^f}{a_{i}^{f}-b_{\min}^{f}} \qquad a_{i}^f\ge b_{j}^f\ge b_{\min}^f
\\
-\omega_{i} \frac{b_{j}^{f}-a_{i}^f}{b_{\max}^{f}-a_{i}^f} \qquad\qquad\qquad b_{j}^f > a_{i}^{f} $\;$ \text{and} $\;$ A^f$\;$\text{is} $\;$ \text{tolerable}
\\
-M \qquad\qquad\qquad\qquad\quad$\;$ b_{j}^f > a_{i}^{f} $\;$ \text{and} $\;$ A^f $\;$ \text{is} $\;$ \text{intolerable}
\end{cases}
\end{equation}

For shipper, the vehicle type $A^3$ is a fixed expectation indicator (see Eq. (\ref{shipper-3}) and see Figure \ref{shipperdeliverypricevehicletype}b). The carrier provides the shipper with various vehicle types, and the shipper requires suitable vehicle type for the cargos to be shipped. Its indicator satisfaction can be formulated as:

\begin{equation}\label{shipper-3}
\varepsilon _{ij}^f=
\begin{cases}
1, \qquad\quad$\;$ b_{j}^f = a_{i}^f
\\
-\varphi _{i} , \qquad b_{j}^f \ne a_{i}^f $\;$ \text{and} $\;$ A^{f} $\;$\text{is} $\;$\text{tolerable}
\\
-M, \qquad   b_{j}^f \ne $\;$ a_{i}^f $\;$ \text{and} $\;$ A^{f} $\;$ \text{is} $\;$ \text{intolerable}
\end{cases}
\end{equation}

\begin{figure}[!htb]
\centering
\includegraphics[width=.8\textwidth]{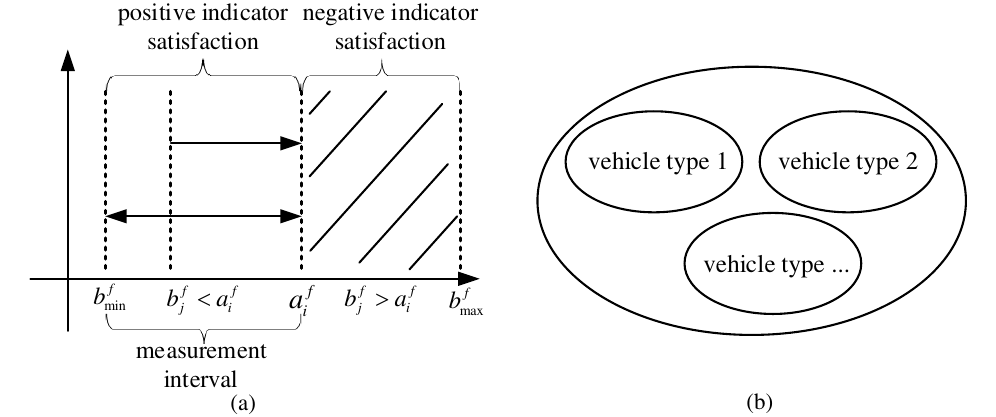} 
\caption{Shipper's indicator satisfaction quantification of the vehicle type}
\label{shipperdeliverypricevehicletype}
\end{figure}

For carrier, the delivery price $B^2$ is a benefit-based expectation indicator (see Eq. (\ref{carrier-1}) and see Figure \ref{carrierdeliverypricevehicletype}a). The higher the shipper's delivery price, the more satisfied the carrier. Its indicator satisfaction can be formulated as:

\begin{equation}\label{carrier-1}
\theta _{ij}^f=
\begin{cases}
\omega_{j}+ \left ( 1-\omega_{j} \right ) \frac{a_{i}^{f}-b_{j}^f}{a_{\max}^{f}-b_{j}^{f}} \qquad b_{j}^f\le a_{i}^f\le a_{\max}^f 
\\
-\omega_{j} \frac{b_{j}^{f}-a_{i}^f}{b_{j}^{f}-a_{\min}^f} \qquad\qquad\qquad$\;$ a_{i}^f < b_{j}^{f} ~ \text{and} ~ B^f ~\text{is tolerable}
\\
-M \qquad\qquad\qquad\qquad\quad$\;$$\;$ a_{i}^f < b_{j}^{f} ~\text{and} ~B^f ~\text{is intolerable}
\end{cases}
\end{equation}

For carrier, the place of delivery $B^8$ is a preference-sequence expectation indicator (see Eq. (\ref{carrier-2}) and see Figure \ref{carrierdeliverypricevehicletype}b). The carrier requires that the shipper's indicator value stays in its own preference-sequence set. Its indicator satisfaction can be formulated as:
\begin{equation}\label{carrier-2}
\theta _{ij}^f=
\begin{cases}
r_{ij}^{-\tau_{j}}, \qquad a_{i}^f $\;$ \text{in} $\;$ b_{j}^f
\\
-M, \qquad a_{i}^f \notin b_{j}^f $\;$ \text{and} $\;$ B^f $\;$ \text{is} $\;$ \text{tolerable} 
\end{cases}
\end{equation}

\begin{figure}[!htb]
\centering
\includegraphics[width=.8\textwidth]{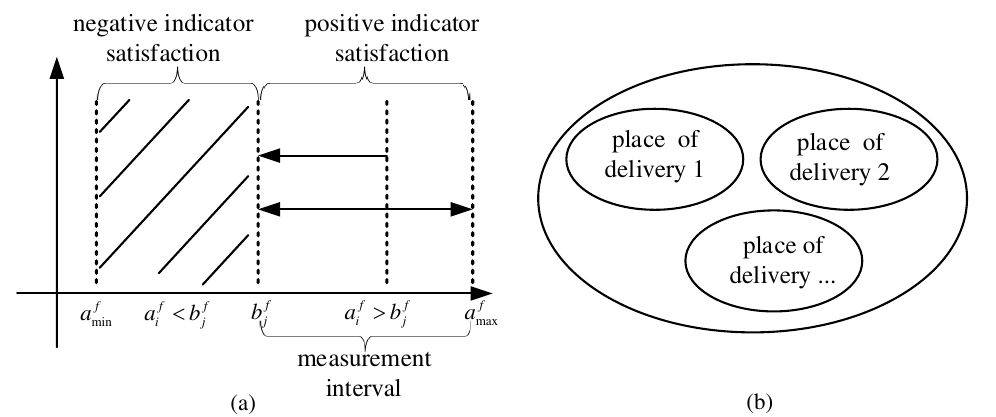} 
\caption{Carrier's indicator satisfaction quantification of the place of delivery}
\label{carrierdeliverypricevehicletype}
\end{figure}

The shipper's reliability indicator evaluation criteria include vehicle completeness, transportation services, and service level. The carrier's reliability indicator evaluation criteria include inventory capacity and service level. $\widehat{A} =[\widehat{a}_{ij}]_{m\times n}$ and $\widehat{B} =[\widehat{b}_{ji}]_{n\times m}$ are the reliability decision matrices from the perspective of the shipper or the carrier, respectively. The elements $\widehat{a}_{ij}$ in $ \widehat{A}$  and the elements  $\widehat{b}_{ji}$ in $\widehat{B}$ are expressed as intuitionistic fuzzy numbers.Through the scoring function, the matrices $\widehat{A} =[\widehat{a}_{ij}]_{m\times n}$ and $\widehat{B} =[\widehat{b}_{ji}]_{n\times m}$ are transformed into the scoring matrices $A=[{a}_{ij}]_{m\times n}$ and $B =[{b}_{ji}]_{n\times m}$, which are the intuitive reflection of the reliability indicator's satisfaction between the two matching sides. Traditional scoring functions typically only incorporate membership and non-membership degree information. However, in order to account for hesitation degree and conformity psychology, this paper utilizes intuitionistic fuzzy cross-entropy to determine the corresponding distribution (see details in Appendix \ref{Appendix A.}). Since the reliability indicator evaluation criteria weight is completely unknown, this paper establishes the following optimization model (see details in Appendix \ref{Appendix B.}), aiming to maximize the comprehensive reliability indicator satisfaction.

\subsection{Formulation}\label{3.3}
Although AHP is a widely used method for weighting indicators, it requires decision-makers to repeatedly judge the relative importance of the indicators, which can lead to consistency problems with the reciprocal binary judgment matrix. An alternative approach is the Relative Superiority Degree of Adjacent Targets (RSDAT) method, which overcomes the inherent limitations of AHP. This method is also straightforward and convenient to use, aligning with the logic of decision-makers (see details in Appendix \ref{Appendix C.}). 

The weight vector of the evaluation indicator $w=(w_{1},w_{2},\cdots,w_{k})^T$ is obtained. So far, the indicator satisfaction of the expectation indicator and the reliability indicator is aggregated, and the comprehensive satisfaction of the shipper and the carrier with regard to the potential matching subjects is obtained:

\begin{equation}\label{shipperdeliverydate}
\begin{cases}
\alpha_{ij}={\textstyle \sum_{f=1}^{9}} w_{i}^{f}\epsilon_{ij}^{f}
\\
\beta_{ij}={\textstyle \sum_{f=1}^{9}} w_{j}^{f}\theta_{ij}^{f}
\end{cases}
\end{equation}                                                               \noindent where $w_{i}^f$ shows the weight of $R^i$ on the indicator $A^f$,  ${\textstyle \sum_{f=1}^{9}} w_{i}^{f}=1,0 \le w_{i}^{f} \le 1 $; $w_{j}^f$ shows the weight of $S^j$ on the indicator $B^f$, ${\textstyle \sum_{f=1}^{9}} w_{j}^{f}=1,0 \le w_{j}^{f} \le 1 $. In particular, the attribute indicator does not participate in the two matching sides' satisfaction aggregation, whose weights are all equal to 0. Based on the above derivation, the objective function is given as follows to maximize the two matching sides' satisfaction:
\begin{equation}
\label{3.11}
\max f_{1}= {\textstyle \sum_{i=1}^{m}}  {\textstyle \sum_{j=1}^{n}} \alpha _{ij}x_{ij}
\end{equation}
\begin{equation}
\max f_{2}= {\textstyle \sum_{i=1}^{m}}  {\textstyle \sum_{j=1}^{n}} \beta _{ij}x_{ij}
\end{equation}

Each shipper is guaranteed a match with a single carrier, while each carrier is guaranteed at most one match with a shipper, hence:
\begin{equation}
{\textstyle \sum_{j=1}^{n}}x_{ij} \le 1 
\end{equation}
\begin{equation}
{\textstyle \sum_{i=1}^{m}}x_{ij} \le 1 
\end{equation}

At the same time, the cargo weight $a_{i}^4$, length $a_{i}^5$, and size $a_{i}^6$ must not exceed vehicle deadweight $b_{j}^4$, length $b_{j}^5$, and volume $b_{j}^6$ respectively. The departure point between the cargo and vehicle shall be consistent, hence:
\begin{equation}
{\textstyle \sum_{i=1}^{m}}a_{i}^4 x_{ij}\le b_{j}^4
\end{equation}
\begin{equation}
{\textstyle \sum_{i=1}^{m}}a_{i}^5 x_{ij}\le b_{j}^5
\end{equation}
\begin{equation}
{\textstyle \sum_{i=1}^{m}}a_{i}^6 x_{ij}\le b_{j}^6
\end{equation}
\begin{equation}
\left( a_{i}^7-b_{j}^7 \right) x_{ij} = 0
\end{equation}

In the model, $x_{ij}=1$ represents that the shipper $R^i$ matches with carrier $S^j$  as a pairing, on the contrary $x_{ij}=0$ , all the variables involved in this model are 0-1 binary variable, hence:
\begin{equation}
\label{3.19}
x_{ij} = 0\; \text{or} \; 1,\quad i=1,2,\dots,m,j=1,2,\dots,n
\end{equation}
 
\subsection{Model improvement and interactive solving}\label{3.4} 
By setting the matching satisfaction threshold to a non-infinitesimal value, a conditional screening process can be conducted on prospective matching subjects based on their attribute indicators (see Eqs. (\ref{3.11} - \ref{3.19})). This process leads to the derivation of a simplified optimization model.
\begin{equation}\label{28}
Lp1
\begin{cases}
\max f_{1}= {\textstyle \sum_{i=1}^{m}}  {\textstyle \sum_{j=1}^{n}} \alpha _{ij}x_{ij}\\
\max f_{2}= {\textstyle \sum_{i=1}^{m}}  {\textstyle \sum_{j=1}^{n}} \beta _{ij}x_{ij}\\

\hbox{s.t.}~\quad$\;${\textstyle \sum_{j=1}^{n}}x_{ij}\le 1 \\
\quad\quad\quad  {\textstyle \sum_{i=1}^{m}}x_{ij} \le 1 \\
\quad\quad\quad  x_{ij}=0 $\;$ \text{or} $\;$ 1\\
\quad\quad\quad  i=1,2,\dots,m,j=1,2,\dots,n
\end{cases}
\end{equation}
                                                    
The simplified optimization model can be formulated as a dual-objective 0-1 linear programming problem, which is commonly transformed into a single-objective problem for technical reasons. However, this paper concentrates on the combination of objective coefficients to account for the complementarity and consistency of the two matching objectives. In practical applications, it can be challenging to precisely and comprehensively describe the decision-making information of both matching sides, and the attainment of unilateral matching objectives can be even more ambiguous. To address this issue, the concept of membership degree and non-membership degree in intuitionistic fuzzy set theory is introduced to quantify the fuzzy target. Now the membership functions 
$ u_{1}\left( f_{1} \left( x  \right) \right ) $ and $ u_{2}\left( f_{2} \left( x  \right) \right ) $  of the objective functions  $f_{1} \left( x  \right)$   and $f_{2} \left( x  \right)$  are defined as:
\begin{equation}\label{29}
u_{1}\left( f_{1} \left( x  \right) \right )=
\begin{cases}
0 \qquad\qquad\quad  f_{1} \left( x  \right) \le f_{1}^L
\\
\frac{f_{1}\left( x \right)-f_{1}^L}{f_{1}^U-f_{1}^L}   \qquad  f_{1}^L \le f_{1} \left( x  \right) \le f_{1}^U
\\
1 \qquad\qquad\quad  f_{1} \left( x  \right) \ge f_{1}^U
\end{cases}
\end{equation}

\begin{equation}\label{30}
u_{2}\left( f_{2} \left( x  \right) \right )=
\begin{cases}
0 \qquad\qquad\quad  f_{2} \left( x  \right) \le f_{2}^L
\\
\frac{f_{2}\left( x \right)-f_{2}^L}{f_{2}^U-f_{2}^L}   \qquad  f_{2}^L \le f_{2} \left( x  \right) \le f_{2}^U
\\
1 \qquad\qquad\quad  f_{2} \left( x  \right) \ge f_{2}^U
\end{cases}
\end{equation}                                                                               
\noindent where $ f_{1}^{U}=\max_{x\in\Omega }f_{1}\left(x\right) $ and $ f_{2}^{U}=\max_{x\in\Omega }f_{2}\left(x\right)$ , representing the highest satisfaction level of the unilateral matching subjects. They are unilateral matching subjects' expected level excluding other influencing factor. While, $ f_{1}^{L}=\min_{x\in\Omega } f_{1}\left(x\right) $  and  $ f_{2}^{L}=\min_{x\in\Omega } f_{2}\left(x\right) $  represent the lowest satisfaction level of the unilateral matching subjects. The membership functions $u_{1}\left( f_{1}\left(x \right) \right) $  and $u_{2}\left( f_{2}\left(x \right) \right) $  reflect the degree which the unilateral matching subjects' satisfaction closes to the expected level, the non-membership functions;  $v_{1}\left( f_{1}\left(x \right) \right) $  and $v_{2}\left( f_{2}\left(x \right) \right) $    reflect the degree which the unilateral matching subjects' satisfaction deviates from the expected level:

\begin{equation}\label{31}
v_{1}\left( f_{1} \left( x  \right) \right )=
\begin{cases}
1 \qquad\qquad\quad  f_{1} \left( x  \right) \le f_{1}^L
\\
\frac{f_{1}^{UL}-f_{1}\left( x \right) }{f_{1}^{UL}-f_{1}^L}   \quad$\;$  f_{1}^L \le f_{1} \left( x  \right) \le f_{1}^{UL}
\\
0 \qquad\qquad\quad  f_{1} \left( x  \right) \ge f_{1}^{UL}
\end{cases}
\end{equation}

\begin{equation}\label{32}
v_{2}\left( f_{2} \left( x  \right) \right )=
\begin{cases}
1 \qquad\qquad\quad  f_{2} \left( x  \right) \le f_{2}^L
\\
\frac{f_{2}^{UL}-f_{2}\left( x \right) }{f_{2}^{UL}-f_{2}^L}   \quad$\;$  f_{2}^L \le f_{2} \left( x  \right) \le f_{2}^{UL}
\\
0 \qquad\qquad\quad f_{2} \left( x  \right) \ge f_{2}^{UL}
\end{cases}
\end{equation}  
          
\noindent where $ f_{1}^{UL}=f_{1}^{U}-\gamma \left( f_{1}^U-f_{1}^L \right) $ , $0<\gamma<1$, parameter $\gamma$ is the unilateral matching subjects' dissatisfaction interval limitation depending on the actual situation, which is taken as 0.2 in this paper. Combining the matching objective's membership degree function with the non-membership degree function, the unilateral matching objective attainment can be measured as a whole (see Figure \ref{Matching objective function}). Hence, the satisfaction functions of $f_{1}\left( x \right) $  and $f_{2}\left( x \right) $ are defined as:

\begin{equation}\label{33}
s_{h} \left ( f_{h} \left ( x \right )  \right ) =u_{h} \left ( f_{h} \left ( x \right )  \right )-v_{h} \left ( f_{h} \left ( x \right )  \right ), \quad h=1,2
\end{equation}

\begin{figure}[!h]
\centering
\includegraphics[width=.6\textwidth]{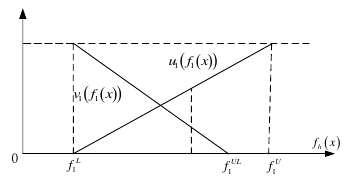}
\caption{Matching objective's membership/non-membership degree (take $f_{1}$ as an example)}
\label{Matching objective function}
\end{figure} 

$s_{1}\left( f_{1}\left(x \right) \right) $ and $s_{2}\left( f_{2}\left(x \right) \right) $ represent the overall satisfaction of the unilateral matching side. An optimization model with the goal of maximizing $s_{1}\left( f_{1}\left(x \right) \right) $ and $s_{2}\left( f_{2}\left(x \right) \right) $  should be:

\begin{equation}\label{34}
Lp2
\begin{cases}
\max \quad s_{1} \left ( f_{1} \left ( x \right )  \right )+s_{2} \left ( f_{2} \left ( x \right )  \right )
\\
\hbox{s.t.}~\quad$\;$ x\in\Omega  \left( a \right) \\ 
\quad\quad\quad  u_{h} \left ( f_{h} \left ( x \right )  \right ) \ge v_{h} \left ( f_{h} \left ( x \right )  \right ) \ge 0, \quad h=1,2  \left( b \right)\\
\quad\quad\quad  u_{h} \left ( f_{h} \left ( x \right )  \right ) + v_{h} \left ( f_{h} \left ( x \right )  \right ) \le 1, \quad h=1,2   \left( c \right)\\
\end{cases}
\end{equation}                                        
                                            
$\Omega $ is the feasible region of $LP1$. For $LP2$, it defines a set of feasible points, denoted as $y$. It also yields an optimal point $y^*\in y$. However, the existence of piecewise functions makes it hard to solve $LP2$ directly. In order to reduce its solving difficulty, we conduct classification discussion, $f_{1}\left( x \right)$ and   $f_{2}\left( x \right)$ are treated as independent variable of $LP2$, then, the value condition of $f_{1}\left( x \right)$ and $f_{2}\left( x \right)$ leads four cases as follows:

Case 1:  $f_{1} {\left( x \right) } \le f_{1}^{UL}$,$f_{2}  {\left( x \right) } \le f_{2}^{UL}$

\begin{equation}\label{shipperdeliverydate}
Lp2-1
\begin{cases}
\max \quad s_{1} \left ( f_{1} \left ( x \right )  \right )+s_{2} \left ( f_{2} \left ( x \right )  \right )
\\
\hbox{s.t.}~\quad$\;$ x\in\Omega  \left( a \right) \\ 
\quad\quad\quad  u_{h} \left ( f_{h} \left ( x \right )  \right ) \ge v_{h} \left ( f_{h} \left ( x \right )  \right ) \ge 0, \quad\quad$\;$$\;$  h=1,2  \left( b \right)\\
\quad\quad\quad  u_{h} \left ( f_{h} \left ( x \right )  \right ) + v_{h} \left ( f_{h} \left ( x \right )  \right ) -1\le 0, \quad h=1,2   \left( c \right)\\
\end{cases}
\end{equation}

Case 2:  $f_{1} {\left( x \right) } \ge f_{1}^{UL}$,$f_{2} {\left( x \right) } \ge f_{2}^{UL}$ 
\begin{equation}\label{shipperdeliverydate}
Lp2-2
\begin{cases}
\max \quad u_{1} \left ( f_{1} \left ( x \right )  \right )+u_{2} \left ( f_{2} \left ( x \right )  \right )
\\
\hbox{s.t.}~\quad$\;$ x\in\Omega  \quad \left( a \right) \\ 
\end{cases}
\end{equation}

Case 3: $f_{1} {\left( x \right) } \le f_{1}^{UL}$, $f_{2} {\left( x \right) } \ge f_{2}^{UL}$ 

\begin{equation}\label{shipperdeliverydate}
Lp2-3
\begin{cases}
\max \quad s_{1} \left ( f_{1} \left ( x \right )  \right )+u_{2} \left ( f_{2} \left ( x \right )  \right )
\\
\hbox{s.t.}~\quad$\;$ x\in\Omega  \left( a \right) \\ 
\quad\quad\quad  u_{1} \left ( f_{1} \left ( x \right )  \right ) \ge v_{1} \left ( f_{1} \left ( x \right )  \right ) \ge 0  \quad\quad$\;$$\;$$\;$ \left( b \right)\\
\quad\quad\quad  u_{1} \left ( f_{1} \left ( x \right )  \right ) + v_{1} \left ( f_{1} \left ( x \right )  \right ) -1\le 0  \quad  \left( c \right)\\
\end{cases}
\end{equation}
  
Case 4: $f_{1} {\left( x \right) } \ge f_{1}^{UL}$, $f_{2} {\left( x \right) } \le f_{2}^{UL}$ 
\begin{equation}\label{shipperdeliverydate}
Lp2-4
\begin{cases}
\max \quad u_{1} \left ( f_{1} \left ( x \right )  \right )+s_{2} \left ( f_{2} \left ( x \right )  \right )
\\
\hbox{s.t.}~\quad$\;$ x\in\Omega  \left( a \right) \\ 
\quad\quad\quad  u_{2} \left ( f_{2} \left ( x \right )  \right ) \ge v_{2} \left ( f_{1} \left ( x \right )  \right ) \ge 0 \quad\quad$\;$$\;$$\;$ \left( b \right)\\
\quad\quad\quad  u_{2} \left ( f_{2} \left ( x \right )  \right ) + v_{2} \left ( f_{2} \left ( x \right )  \right ) -1\le 0  \quad \left( c \right)\\
\end{cases}
\end{equation}

\begin{proposition}
The optimal solution of $LP2$ appears in case 2.
\end{proposition}
\begin{proof}
    Inspired by the prerequisite of case 1 $\sim$ case 4, we take $\left \{ f_{1}(x) \right|f_{1}(x)=f_{1}^{UL} \}   $and $ \left \{ f_{1}(x) \right| u_{1}(f_{1}(x))=v_{1}(f_{1}(x)) \}$  as the horizontal axis demarcation point to divide the area of the Figure \ref{Matching objective function}, aiming to observe the solution of $LP2$ (see Figure \ref{partitioning}). With the existence of $\gamma \left ( \gamma > 0 \right)$, it is easy to know that $y$ appears in part I or part II combined with constraint condition $ \left( b \right)$ . The worst solution of part II is better than the optimal solution of part I, that is, the optimal solution of $LP2$ appears in case 2.
    \end{proof}
\begin{figure}[!h]
\centering
\includegraphics[width=.6\textwidth]{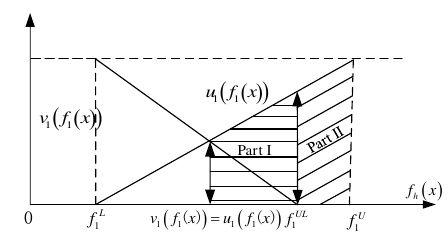}
\caption{The partitioning of Figure 5 (take $f_{1}$ as an example)}
\label{partitioning}
\end{figure} 
               
After finishing the above derivation, we turn to $LP3$:

\begin{equation}\label{35}
Lp3
\begin{cases}
\max \quad u_{1} \left ( f_{1} \left ( x \right )  \right )+u_{2} \left ( f_{2} \left ( x \right )  \right )
\\
\hbox{s.t.}~\quad x\in\Omega   \\ 
\quad\quad$\;$$\;$ f_{1}\left( x \right) \ge f_{1}^{UL}\\ 
\quad\quad$\;$$\;$ f_{2}\left( x \right) \ge f_{2}^{UL}
\end{cases}
\end{equation}

In this paper, the interactive satisfactory decision-making method is used to meet fairness factor's pre-set value interval. In detail, the model is gradually improved on the basis of the local effective solution until the final satisfactory solution is obtained. Let $\eta =s_{1}(f_{1}(x))/s_{2}(f_{2}(x))$  be the fairness measurement standard, the value interval $\left [ n_{L},n_{H} \right ]$   of $n$  is set in advance by the online vehicle-cargo matching platform. Solve $LP3$ , and if the obtained effective solution $x^*$  can satisfy $n \in \left [ n_{L},n_{H} \right ]$ ,  $x^*$ is the final satisfactory solution, otherwise a new effective solution will be found again, the strategy is:

1) When $ \eta <\eta_{L} $ , let the membership degree level of the matching side whose objective function is $f_{1}\left( x \right) $ raise from $\theta_{1}$ to $\bar{\theta _{1} }$ ;

2) When $ \eta >\eta_{H} $ , let the membership degree level of the matching side whose objective function is  $f_{2}\left( x \right) $ raise from $\theta_{2}$ to $\bar{\theta _{2} }$  ;

3) Add constraint $\mu _{h} \left ( f_{h}\left ( x \right )  \right ) \ge \bar{\theta _{h}} $ on the basis of $LP3$ and resolve it.  

If the new effective solution can meet the fairness measurement standard, it will be output as the overall satisfactory solution, then the decision-making process stops; otherwise, it will be updated according to the above strategy until the obtained effective solution meets that (see Table \ref{Algorithm1}).

\begin{table}[!h]
\caption{Algorithm1}\label{Algorithm1}
\resizebox{\textwidth}{1.8in}{
\begin{tabular}{|cl|}
\hline
Algorithm 1: & Obtaining satisfactory solution by adding fairness adjustment constraint                                                                                             \\ \hline
Input:       & $\left [ n_{L},n_{H} \right ]$  and $\gamma$                                                                                                                         \\ 
Output:      & Satisfactory matching scheme                                                                                                                                         \\ 
            & Begin                                                                                                                                                                \\ 
Step1:       & \begin{tabular}[c]{@{}l@{}}Set the fairness measurement interval of the matching platform $\left [ n_{L},n_{H} \right ]$  \\ and the value of  $\gamma$\end{tabular} \\ 
Step2:       & Solve $LP3$                                                                                                                                                          \\ 
Step3:       & Define $u_{1} \left ( f_{1} \left ( x \right )  \right )=\theta_{1} $ , $u_{2} \left ( f_{2} \left ( x \right )  \right )=\theta _{2}$                         \\ 
            & Define  $\eta=u_{1} \left ( f_{1} \left ( x \right )  \right ) / u_{2} \left ( f_{2} \left ( x \right )  \right )$                                                   \\ 
Step4:           & If $\eta_{L}\le \eta \le \eta_{H}$ then                                                                                                                              \\ 
       & \quad End                                                                                                                                             \\ 
       & Else                                                                                                                                                                 \\ 
       & \quad If $\eta \le \eta_{L}$ then                                                                                                                     \\ 
      & \quad\quad\quad  Add constraint $u_{1} \left ( f_{1} \left ( x \right )  \right ) \ge \bar{\theta _{1}}$                \\ 
      & \quad  Else                                                                                                                                           \\ 
      & \quad\quad\quad   Add constraint $u_{2} \left ( f_{1} \left ( x \right )  \right ) \ge \bar{\theta _{2}}$              \\ 
      & \quad End if                                                                                                                                         \\ 
      & End if                                                                                                                                                               \\ 
      & Return Step2                                                                                                                                                         \\ 
\hline
\end{tabular}}
\end{table}

Note: Algorithm1 adjusts the optimal matching scheme (i.e., local effective solution) obtained from $LP3$ by adding fairness adjustment constraint, that is, the satisfactory matching scheme not only considers fairness concern but also has high satisfaction.

\section{Evolutionary strategy analysis in the vehicle-cargo matching process with fairness concern}\label{4.0}
In this section, we model the realistic vehicle-cargo matching dynamic evolution process with fairness concerns and derive general results regarding such a process.
\subsection{Basic assumptions}\label{6.1}
According to the above analysis, the online platform can formulate a satisfactory vehicle-cargo matching scheme tailored to real-world circumstances. However, the successful implementation of this recommended scheme hinges on the strategic selections made by shippers and carriers. As key stakeholders in the formulation and execution of vehicle-cargo matching plans, shippers, carriers, and the online platform all possess bounded rationality and seek to maximize their individual interests by adjusting their strategic decisions. Assuming the platform's choices involve offering subsidies or not, while shippers and carriers decide whether to accept or reject the proposed matching plan, a tripartite evolutionary game model is constructed to explore these dynamics (see Table \ref{Table16}).

	\setlength\LTleft{0pt}
	\setlength\LTright{0pt}
	\begin{longtable}{@{\extracolsep{\fill}}cl}
 
		\caption{Mathematical symbols}	\label{Table16} \\
			\hline
			\multicolumn{1}{c}{\begin{tabular}[c]{@{}c@{}}Constants \\or parameter\end{tabular}} & \multicolumn{1}{c}{\begin{tabular}[c]{@{}c@{}}Explanation\end{tabular}}  \\
			\hline
			$x$    & The probability of a shipper accepting the recommended matching scheme
			\\
			$y$      & The probability of a carrier accepting the recommended matching scheme
			\\
			$z$      & The probability of the platform adopting "subsidy" strategy
			\\
   \hline 
    \multicolumn{2}{l}{\begin{tabular}[l]{@{}l@{}}  If the carrier selects "acceptance" strategy while the shipper selects "acceptance" strategy  \end{tabular}} 
   \\ \hline
			$u_{I}R_{I}$ &  Direct benefits obtained by shipper ($u_{I}$ is the service level coefficient)
			\\
			$hW_{I}$ &  {\begin{tabular}[l]{@{}l@{}} Order commission costs paid to platform  by shipper ($h$ is the percentage \\of commission)\end{tabular}}
			\\
			$u_{P}R_{P}$  & Direct benefits obtained by carrier ($u_{P}$ is the service level coefficient) 
			\\
            $fW_{P}$  &  {\begin{tabular}[l]{@{}l@{}} Order commission costs paid to platform  by carrier ($f$ is the percentage \\of commission)\end{tabular}}
			\\
    \hline 
    \multicolumn{2}{l}{\begin{tabular}[l]{@{}l@{}}  If the carrier selects "non-acceptance" strategy while the shipper selects "acceptance" \\ strategy  \end{tabular}} 
   \\ \hline
            $q_{1}$  &  {\begin{tabular}[l]{@{}l@{}} The shipper will find a matching subject by itself, the probability of \\ obtaining  a higher return $R_{I}^U$ than the matching subject recommended \\  by platform \end{tabular}}
			\\
            $q_{2}$ & {\begin{tabular}[l]{@{}l@{}} The shipper will find a matching subject by itself, the probability of\\ obtaining a lower return $R_{I}^L$ than the matching subject recommended \\  by platform \end{tabular}}
            \\ 
            $C_{I}$ & {\begin{tabular}[l]{@{}l@{}}The corresponding time cost of waiting for the carrier's response  paid \\ by shipper\end{tabular}}
			\\
            $\sigma_{1}$ & The carrier shall bear the $\sigma_{1}$ proportion of $C_{I}$
			\\
    		$Q_{I}$     &  {\begin{tabular}[l]{@{}l@{}} The follow-up cost of finding new matching subject for contact and \\ communication  paid by shipper\end{tabular}}
	        \\
			 $p_{1}$     &  {\begin{tabular}[l]{@{}l@{}} The carrier will find a matching subject by itself, the probability of \\obtaining a higher return $R_{P}^U$ than the matching subject recommended \\ by platform \end{tabular}}
            \\
			 $p_{2}$     &  {\begin{tabular}[l]{@{}l@{}} The carrier will find a matching subject by itself, the probability of\\ obtaining a lower return $R_{P}^L$ than the matching subject recommended \\ by platform\end{tabular}}
            \\
            $Q_{P}$     & {\begin{tabular}[l]{@{}l@{}}The follow-up cost of finding new matching subject for contact and \\ communication paid by carrier\end{tabular}}
	        \\
    \hline 
    \multicolumn{2}{l}{\begin{tabular}[l]{@{}l@{}}  If the shipper selects "acceptance" strategy while the carrier selects "acceptance" \\ strategy  \end{tabular}} 
   \\ \hline
           $u_{P}R_{P}$ &  Direct benefits obtained by carrier ($u_{P}$ is the service level coefficient)
			\\
			$fW_{P}$ &  {\begin{tabular}[l]{@{}l@{}} Order commission costs paid to platform  by carrier ($f$ is the \\ percentage of commission)\end{tabular}}
			\\
			$u_{I}R_{I}$  & Direct benefits obtained by shipper ($u_{I}$ is the service level coefficient)
			\\
            $hW_{I}$  &  {\begin{tabular}[l]{@{}l@{}} Order commission costs paid to platform  by shipper ($h$ is the \\ percentage of commission)\end{tabular}}
			\\
    \hline 
    \multicolumn{2}{l}{\begin{tabular}[l]{@{}l@{}}  If the shipper selects "non-acceptance" strategy while the carrier selects "acceptance" \\ strategy  \end{tabular}} 
   \\ \hline
            $p_{1}$  &  {\begin{tabular}[l]{@{}l@{}} The carrier will find a matching subject by itself, the probability of\\ obtaining a higher return $R_{P}^U$ than the matching subject recommended \\by platform \end{tabular}}
			\\
            $p_{2}$ & {\begin{tabular}[l]{@{}l@{}} The carrier will find a matching subject by itself, the probability of \\ obtaining a lower return $R_{I}^L$ than the matching subject recommended \\ by platform \end{tabular}}
            \\
            $C_{P}$ & {\begin{tabular}[l]{@{}l@{}} The corresponding time cost of waiting for the shipper's response paid \\ by carrier\end{tabular}}
			\\
            $\sigma_{2}$ & The shipper shall bear the $\sigma_{2}$ proportion of $C_{P}$
			\\
    		$Q_{P}$     & {\begin{tabular}[l]{@{}l@{}}The follow-up cost of finding new matching subject for contact and \\ communication paid by carrier\end{tabular}}
	        \\
			 $q_{1}$     &  {\begin{tabular}[l]{@{}l@{}} The shipper will find a matching subject by itself, the probability of \\obtaining a higher  return $R_{I}^U$ than the matching subject \\recommended by platform \end{tabular}}
            \\
			 $q_{2}$     &  {\begin{tabular}[l]{@{}l@{}} The shipper will find a matching subject by itself, the probability of\\ obtaining a lower return $R_{I}^L$ than the matching subject \\recommended by platform  \end{tabular}}
            \\
    		$Q_{I}$     &  {\begin{tabular}[l]{@{}l@{}}The follow-up cost of finding new matching subject for contact and \\ communication paid by shipper\end{tabular}}
	\\
    \hline 
    \multicolumn{2}{l}{\begin{tabular}[l]{@{}l@{}}  When the platform selects "subsidy"  strategy  \end{tabular}} 
   \\ \hline
			$\alpha$ & \begin{tabular}[l]{@{}l@{}}  Platform's subsidy intensity to shipper when shipper selects "acceptance" \\ strategy \end{tabular}
			\\
			$S_{I}$   &{\begin{tabular}[l]{@{}l@{}} Platform's maximum subsidy amount to shipper when shipper selects \\ "acceptance" strategy  \end{tabular}} 
			\\
			$\beta$    & {\begin{tabular}[l]{@{}l@{}} Platform's subsidy intensity to carrier when carrier selects "acceptance" \\ strategy \end{tabular}} 
			\\
			$S_{P}$  &{\begin{tabular}[l]{@{}l@{}} Platform's maximum subsidy amount to carrier when carrier selects \\ "acceptance"   strategy  \end{tabular}} 
			\\
			$D_{G}$    & {\begin{tabular}[l]{@{}l@{}} Management costs when platform formulates and implements the subsidy \\ policy\end{tabular}} 
			\\
			$F_{G}^I$ & {\begin{tabular}[l]{@{}l@{}} The reputation revenue obtained by the platform when subsidizing the \\ shipper \end{tabular}} 
            \\
            $F_{G}^P$ &{\begin{tabular}[l]{@{}l@{}}  The reputation revenue obtained by the platform when subsidizing the \\ carrier \end{tabular}} 
            \\
            $\eta$ & Fairness factor
            \\
            \hline
	\end{longtable}

The income matrix corresponding to stakeholders' various strategic selections is as follows:

\begin{table}[h!]
\begin{center}
\caption{Income matrix}\label{incometable}
\resizebox{1.\textwidth}{1.2in}{
\begin{tabular}{cccc}
\hline
\multirow{2}{*}{Shipper} & \multirow{2}{*}{Carrier} & \multicolumn{2}{c}{Platform} \\ & & subsidy ($z$) & non-subsidy (1-$z$) \\
\hline

\multirow{6}{*}{\begin{tabular}[c]{@{}c@{}}acceptance\\ ($x$)\end{tabular}} & \multirow{3}{*}{acceptance ($y$)} & $u_{I}R_{I}+\alpha S_{I}-hW_{I}$ & $u_{I}R_{I}-hW_{I}$ \\
 & & $u_{p}R_{p}+\beta S_{p}-fW_{p}$ & $u_{p}R_{p}-fW_{p}$ \\
& & $\eta(F_{G}^p+F_{G}^I)- \alpha s_{I}-\beta S_{p}-D_{G}$ & 0 \\

 & \multirow{3}{*}{\begin{tabular}[c]{@{}c@{}}non-\\ acceptance\\ (1-$y$)\end{tabular}} &  $(q_{1}R_{I}^U+q_{2}R_{I}^L)-(1-\sigma_{1})C_{I}-Q_{I}+\alpha S_{I}$ & $(q_{1}R_{I}^U+q_{2}R_{I}^L)-(1-\sigma_{1})C_{I}-Q_{I}$ \\
 & & $(p_{1}R_{p}^u+p_{2}R_{p}^l)-\sigma_{1}C_{I}-Q_{p}$ & $(p_{1}R_{p}^U+p_{2}R{p}^L)-\sigma_{1}C_{I}-Q_{p}$ \\
 & & $\eta F_{G}^I-\alpha S_{I}-D_{G}$ & 0 \\
 
\multirow{6}{*}{\begin{tabular}[c]{@{}c@{}}non-\\ acceptance\\ (1-$x$)\end{tabular}} & \multirow{3}{*}{acceptance ($y$)} & $(q_{1}R_{I}^U+q_{2}R_{I}^L)-\sigma_{2}C_{P}-Q_{I}$ & $(q_{1}R_{I}^u+q_{2}R_{I}^l)-\sigma_{2}C_{P}-Q_{I}$ \\
 & &  $(p_{1}R_{P}^U+p_{2}R_{P}^L)-(1-\sigma_{2})C_{P}-Q_{P}+\beta S_{P}$ & $(p_{1}R_{P}^U+p_{2}R_{P}^L)-(1-\sigma_{2})C_{P}-Q_{P}$ \\
 & & $\eta F_{G}^P-\beta S_{p}-D_{G}$ & 0 \\
 
 & \multirow{3}{*}{\begin{tabular}[c]{@{}c@{}}non-\\ acceptance\\ (1-$y$)\end{tabular}} & $(q_{1}R_{I}^U+q_{2}R_{I}^L)-Q_{I}$ & $(q_{1}R_{I}^U+q_{2}R_{I}^L)-Q_{I}$ \\
 & & $(p_{1}R_{P}^U+p_{2}R_{P}^L)-Q_{P}$ & $(p_{1}R_{P}^U+p_{2}R_{P}^L)-Q_{P}$ \\
 & & $-D_{G}$ & 0
 \\ \hline
\end{tabular}}
\end{center}
\end{table}

According to the Table \ref{incometable}, the dynamic replication equations of different game players can be calculated:

The dynamic replication equation of the shipper:
\begin{equation}
\resizebox{.9\hsize}{!}{$F\left( x \right) =x\left(1-x \right) \left( -C_{I}+C_{I}\sigma_{1}+\left(C_{I}-hW_{I}+Q_{I}-q_{2}R_{I}^L-q_{1}R_{I}^U+R_{I}u_{I}-C_{I}\sigma_{1}+C_{P}\sigma_{2} \right)y+\alpha S_{I}z\right)$}
\end{equation}

The dynamic replication equation of carrier:
\begin{equation}
\resizebox{.9\hsize}{!}{$F\left( y \right) =y\left(1-y \right) \left( -C_{P}+C_{P}\sigma_{2}+\left(C_{P}-fW_{P}+Q_{P}-p_{2}R_{P}^L-p_{1}R_{P}^U+R_{P}u_{P}+C_{I}\sigma_{1}-C_{P}\sigma_{2} \right)x+\beta S_{P}z\right)$}
\end{equation}      

The dynamic replication equation of the platform:
\begin{equation}
\resizebox{.5\hsize}{!}{$F\left( z \right ) =z\left(1-z \right) \left( -D_{G}+x\eta F_{G}^I+y \eta F_{G}^P-x\alpha S_{I}-y\beta S_{P} \right)$}
\end{equation}
                                                   
\subsection{Solution}
The dynamic system evolution and stability strategy can be obtained through the analysis of the Jacobian matrix. According to the dynamic replication equation of the three stakeholders, the Jacobian matrix of the replication dynamic system is:
\begin{equation}
\phi=(\phi_1, \phi_2, \phi_3)
\end{equation}
\noindent where 
$$
\phi_1=
\begin{bmatrix}
 (1-2x)\left ( -C_{I}+C_{I}\sigma_{1}+\binom{C_{I}-hW_{I}+Q_{I}
-q_{2}R_{I}^L-}{q_{1}R_{I}^U +R_{I}u_{I}-C_{I}\sigma_{1}+C_{P}\sigma_{2}} y+\alpha S_{I}z \right )
\\
y(1-y)\binom{C_{P}-fW_{P}+Q_{P}-p_{2}R_{P}^L-}{p_{1}R_{P}^U+R_{P}U_{P}+C_{I}\sigma_{1}-C_{P}\sigma_{2}}
\\
z(1-z)(\eta F_{G}^I-\alpha S_{I})
\end{bmatrix}
$$
$$
\phi_2=
\begin{bmatrix}
x(1-x)\binom{C_{I}-hW_{I}+Q_{I}-q_{2}R_{I}^L-}{q_{1}R_{I}^U+R_{I}u_{I}-C_{I}\sigma_{1}+C_{P}\sigma_{2}}
\\
(1-2y)\left ( -C_{P}+C_{P}\sigma_{2}+\binom{C_{P}-fW_{P}+Q_{P}
-p_{2}R_{P}^L-}{p_{1}R_{P}^U +R_{P}U_{P}+C_{I}\sigma_{1}-C_{P}\sigma_{2}} x+\beta S_{P}z \right )
\\
z(1-z)(\eta F_{G}^P-\beta S_{P}) 
\end{bmatrix}
$$
$$
\phi_3=
\begin{bmatrix}
x(1-x)\alpha S_{I} 
\\
y(1-y)\beta S_{P} 
\\
(1-2z)\binom{-D_{G}+x\eta F_{G}^I+}{y\eta F_{G}^P-x\alpha S_{I}-y\beta S_{P}}  
\end{bmatrix}
$$

Let the replicated dynamic equations satisfy the relationship $F(x)=F(y)=F(z)=0$, and we get 8 equilibrium points, that is $E_{1}(0,0,0),E_{2}(0,0,1),E_{3}(0,1,0),E_{4}(0,1,1),E_{5}(1,0,0)$, $E_{6}(1,0,1),E_{7}(1,1,0),E_{8}(1,1,1)$. The equilibrium points' Jacobian matrix eigenvalues can be obtained by substituting the above 8 pure strategy equilibrium points into the dynamic system Jacobian matrix respectively (see Table \ref{Jacobiantable}):

\begin{table}[h]
\centering
\caption{Eigenvalues of Jacobian matrix} \label {Jacobiantable}
{\scriptsize
\begin{tabular}{ccccc}
\hline
Category & \begin{tabular}[c]{@{}l@{}}Equilibrium point\end{tabular} & Eigenvalue $\lambda_1$    & Eigenvalue $\lambda_2$ & Eigenvalue $\lambda_3$ \\ \hline
$E_{1}$ & 0,0,0  & $ -C_{I}+C_{I} \sigma_{1}$                                 &  $ -C_{P}+C_{P} \sigma_{2}$              &    $ -D_{G}$   \\

$E_{2}$ & 0,0,1  & $-C_{I}+C_{I}\sigma_{1}+\alpha S_{I} $                     &  $ -C_{P}+C_{P} \sigma_{2}+\beta S_{P}$  &   $D_{G} $   \\

$E_{3}$ & 0,1,0  & \begin{tabular}[l]{@{}l@{}} $-hW_{I}+Q_{I}-q_{2}R_{I}^L-$\\
                   $q_{1}R_{I}^U+R_{I}u_{I}+C_{P}\sigma_{2}$ \end{tabular}     &   $C_{P}-C_{P}\sigma_{2}$  & $-D_{G}+\eta F_{G}^P-\beta S_{P}$  \\
                   
$E_{4}$ & 0,1,1  & \begin{tabular}[l]{@{}l@{}} $-hW_{I}+Q_{I}-q_{2}R_{I}^L-$\\
                   $q_{1}R_{I}^U+R_{I}u_{I}+C_{P}\sigma_{2}+\alpha S_{I}$ \end{tabular}  &  $ C_{P}-C_{P}\sigma_{2}-\beta S_{P} $  &  $D_{G}-\eta F_{G}^P+\beta S_{P}$           \\
                   
$E_{5}$ & 1,0,0  & $ C_{I}-C_{I} \sigma_{1} $     &   \begin{tabular}[l]{@{}l@{}} $-fW_{P}+Q_{p}-p_{2}R_{P}^L-$ \\ 
                                                        $p_{1}R_{P}^U+R_{P}u_{P}+C_{I}\sigma_{1}$ \end{tabular} &  $-D_{G}+\eta F_{G}^I-\alpha S_{I}$  \\
                                                        
$E_{6}$ & 1,0,1  & $ C_{I}-C_{I} \sigma_{1}-\alpha S_{I} $      &  \begin{tabular}[l]{@{}l@{}} $-fW_{P}+Q_{p}-p_{2}R_{P}^L-$ \\ 
                                                                     $p_{1}R_{P}^U+R_{P}u_{P}+C_{I}\sigma_{1}+\beta S_{P}$ \end{tabular}   &  $D_{G}-\eta F_{G}^I+\alpha S_{I}$    \\
                                                                     
$E_{7}$ & 1,1,0  & \begin{tabular}[l]{@{}l@{}} $hW_{I}-Q_{I}+q_{2}R_{I}^L+$\\
                   $q_{1}R_{I}^U-R_{I}u_{I}-C_{P}\sigma_{2}$ \end{tabular}      &    \begin{tabular}[l]{@{}l@{}} $fW_{P}-Q_{p}+p_{2}R_{P}^L+$ \\ 
                                                        $p_{1}R_{P}^U-R_{P}u_{P}-C_{I}\sigma_{1}$ \end{tabular}  &   \begin{tabular}[l]{@{}l@{}}$-D_{G}+\eta F_{G}^I+$\\$\eta F_{G}^P-\alpha S_{I}-\beta S_{P} $  \end{tabular}    \\
                                                        
$E_{8}$ & 1,1,1  & \begin{tabular}[l]{@{}l@{}} $hW_{I}-Q_{I}+q_{2}R_{I}^L+$\\
                   $q_{1}R_{I}^U-R_{I}u_{I}-C_{P}\sigma_{2}-\alpha S_{I}$ \end{tabular}   &   \begin{tabular}[l]{@{}l@{}} $fW_{P}-Q_{p}+p_{2}R_{P}^L+$ \\ 
                                                        $p_{1}R_{P}^U-R_{P}u_{P}-C_{I}\sigma_{1}-\beta S_{P} $ \end{tabular}  &  \begin{tabular}[l]{@{}l@{}}$D_{G}-\eta F_{G}^I-$\\ $\eta F_{G}^P+\alpha S_{I}+\beta S_{P} $  \end{tabular}  \\  
\hline
\end{tabular} }
\end{table}

According to Lyapunov's first method, if all the eigenvalues of the Jacobian matrix are negative, the system's equilibrium point is evolutionarily stable. When parameter variables change under different conditions, the equilibrium point's eigenvalues and the dynamic system's evolutionary path change accordingly, resulting in significant fluctuations in the stability of the equilibrium point. This paper focuses on two ideal stable states (1,1,0) and (1,1,1), determining the corresponding establishment conditions for the existence of a unique stable equilibrium point in the dynamic system, as shown in Table \ref{Conditionstab}:

\begin{table}[!htbp]
\begin{center}
\caption{Conditions for the existence of a unique stable equilibrium point (1,1,0) or  (1,1,1)}\label{Conditionstab}
{\scriptsize
\begin{tabular}{lll}
\hline
                            & A unique stable equilibrium point(1,1,0)                                                                                                                      & A unique stable equilibrium point (1,1,1)
                            \\ \hline
\multirow{4}{*}{Conditions} & $Q_{I} > hW_{I}+q_{2}R_{I}^L+q_{1}R_{I}^U-R_{I}u_{I}-C_{P}\sigma_{2}$ & $Q_{I} > hW_{I}+q_{2}R_{I}^L+q_{1}R_{I}^U-R_{I}u_{I}-C_{P}\sigma_{2}-\alpha S_{I}$ \\
                            & $Q_{P} > fW_{P}+p_{2}R_{P}^L+p_{1}R_{P}^U-R_{P}u_{P}-C_{I}\sigma_{1}$ & $Q_{P} > fW_{P}+p_{2}R_{P}^L+p_{1}R_{P}^U-R_{P}u_{P}-C_{I}\sigma_{1}-\beta S_{P}$ \\
                            & $ D_{G}>\eta F_{G}^{I}+\eta F_{G}^{P}-\alpha S_{I} -\beta S_{P}$     &    $ D_{G}<\eta F_{G}^{I}+\eta F_{G}^{P}-\alpha S_{I} -\beta S_{P}$    \\
                            &  $C_{I}-C_{I}\sigma_{1}=0 \quad \text{or} \quad C_{P}-C_{P} \sigma_{2}=0$        &      $C_{I}-C_{I}\sigma_{1}=0 \quad \text{or} \quad C_{P}-C_{P}\sigma_{2}=0$ \\
\hline
\end{tabular}}
\end{center}
\end{table}
It should be noted that, to achieve a unique stable point in the dynamic system, one of the following conditions must be met: either $\sigma _{1}=1$ or $\sigma _{2}=1$. In the former case, the carrier fully bears the shipper's time cost $C_{I}$ associated with waiting for a response when the carrier does not accept the recommended scheme but the shipper does. In the latter case, the shipper fully bears the carrier's time cost loss $C_{P}$ incurred while waiting for a response when the shipper does not accept the recommended matching scheme but the carrier does.

\section{Numerical experiment and simulation analysis}\label{5.0} 
In this section, we conduct numerical experiments and simulation analyses to verify the analytical model and explore practical implications when considering fairness concerns in vehicle-cargo matching (see Figure \ref{flow chart}).
\begin{figure}[!htbp]
\centering
\includegraphics[width=.9\textwidth]{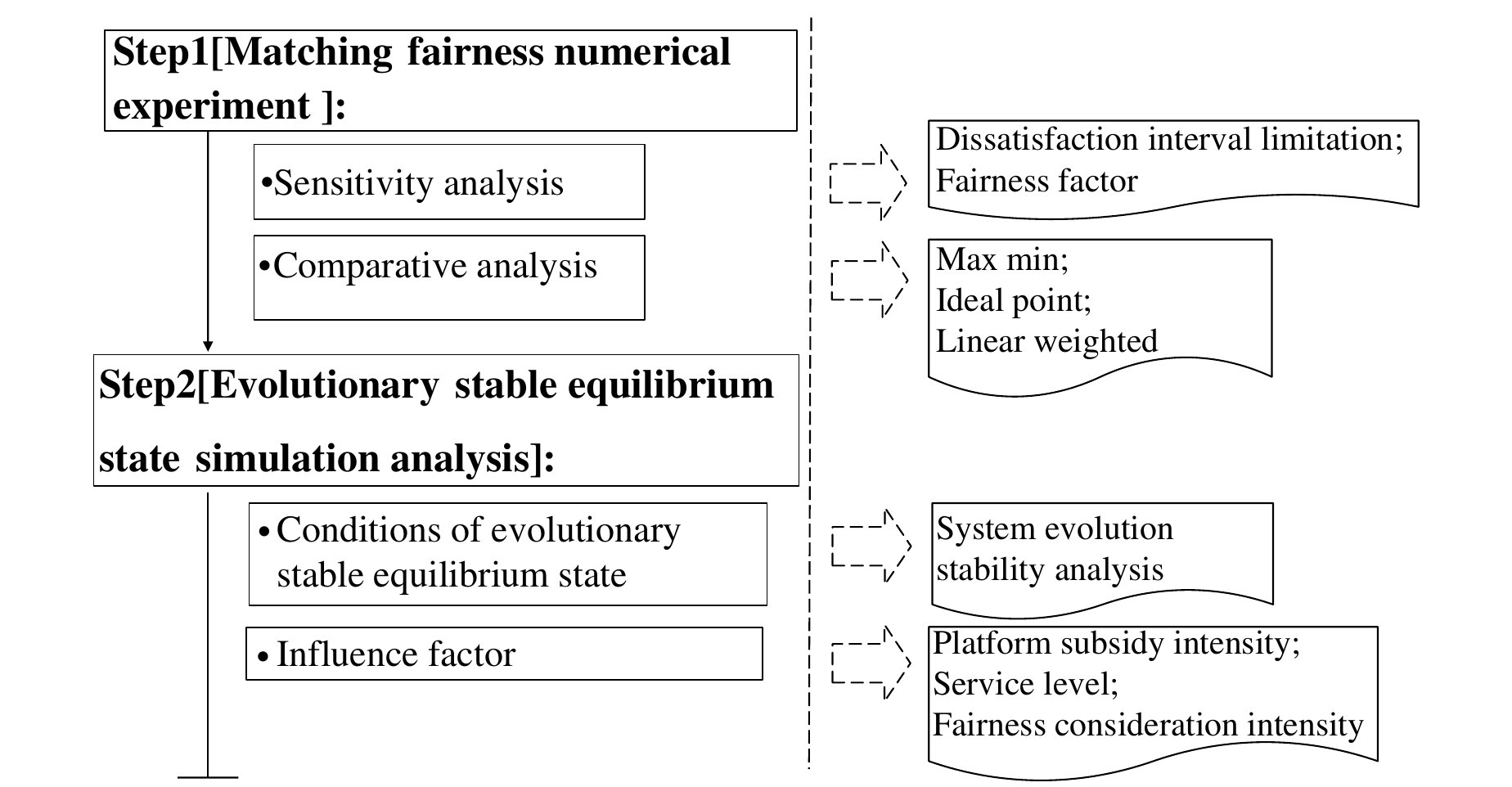}
\caption{Flow chart of numerical experiment and simulation analysis}
\label{flow chart}
\end{figure}

\pagebreak

\subsection{Numerical experiment} 
After a pre-conditional screening process, an online vehicle-cargo matching platform needs to match 8 shippers and 10 carriers in the entire vehicle transportation mode. The necessary information on expectation and reliability indicators is shown in Table \ref{Shippertab} - Table \ref{Carrierreliabilitytab} (Y represents the expectation indicator is tolerable, and N represents the expectation indicator is intolerable, specific indicator information is shown in the form of code for simplicity, i.e., $A^{3}$, $A^{8}$, $B^{3}$ and $B^{8}$), and the pre-set value interval $\eta$ of is $\left[ 0.75,1 \right]$:

\begin{table}[!htbp]
\begin{center}
\caption{Shippers' expectation indicator information}\label{Shippertab}
\resizebox{.65\textwidth}{.8in}{
\begin{tabular}{cccccc}
\hline
         & $A^{1}$ & $A^{2}$ & $A^{3}$ & $A^{8}$ & $A^{9}$ \\
         \hline
Shipper1 & {[}1,4{]} N               & 50 Y                      & Vehicle type* Y          & Location*                & -                       \\ 
Shipper2 & {[}2,6{]} Y               & 60 N                      & Vehicle type@ Y          & Location@               & -                        \\ 
Shipper3 & {[}1,3{]} N               & 75 N                      & Vehicle type@ Y          & Location*                & -                         \\ 
Shipper4 & {[}1,2{]} Y               & 55 N                      & Vehicle type@ Y          & Location$\epsilon$                & -                        \\ 
Shipper5 & {[}1,4{]} N               & 50 Y                      & Vehicle type* Y          & Location*                & -                        \\
Shipper6 & {[}2,6{]} Y               & 60 N                      & Vehicle type@ Y          & Location@               &-                         \\ 
Shipper7 & {[}1,3{]} N               & 75 N                      & Vehicle type@ Y          & Location*                & -                        \\ 
Shipper8 & {[}1,2{]} Y               & 55 N                      & Vehicle type@ Y          & Location$\epsilon$               & -                        \\ \hline          
\end{tabular}}
\end{center}
\end{table}

\begin{table}[!htbp]
\begin{center}
\caption{Carriers' expectation indicator information}\label{Carriertab}
\resizebox{.95\textwidth}{.95in}{
\begin{tabular}{cccccc}
\hline
 & $B^1$ & $B^2$ & $B^3$ & $B^8$ & $B^9$\\
 \hline
Carrier1  & 2 & 45 N & Vehicle type@ & Location$\Psi$,Location@,Location*,Location$\epsilon$  N &  -
\\
Carrier2  & 4 & 40 N & Vehicle type@ & Location@,Location$\epsilon$ ,Location* N           &  -
\\
Carrier3  & 1 & 50 Y & Vehicle type@ & Location$\epsilon$ ,Location$\Theta$,Location@,Location* N & -
      \\
Carrier4  & 3 & 35 N & Vehicle type* & Location$\epsilon$ ,Location$\wp$,Location@,Location* N & -                     \\
Carrier5  & 3 & 40 N & Vehicle type* & Location$\epsilon$ ,Location*,Location@ N           & -                     \\
Carrier6  & 2 & 45 N & Vehicle type@ & Location$\Psi$,Location@,Location*,Location$\epsilon$  N & -                     \\
Carrier7  & 4 & 40 N & Vehicle type@ & Location@,Location$\epsilon$ ,Location* N           & -                     \\
Carrier8  & 1 & 50 Y & Vehicle type@ & Location$\epsilon$ ,Location$\Theta$,Location@,Location* N & -                     \\
Carrier9  & 3 & 35 N & Vehicle type* & Location$\epsilon$ ,Location$\wp$,Location@,Location* N & -                     \\
Carrier10 & 3 & 40 N & Vehicle type* & Location$\epsilon$ ,Location*,Location@ N           & -     \\
\hline
\end{tabular}}
\end{center}
\end{table}

\begin{table}[!htbp]
\caption{Shipper's reliability indicator matrix for evaluating carriers}\label{Shipperreliabilitytab}
\resizebox{\textwidth}{.7in}{
\begin{tabular}{ccccccccccc}
\hline 
 & Carrier1 & Carrier2 & Carrier3 & Carrier4 & Carrier5 & Carrier6 & Carrier7 & Carrier8 & Carrier9 & Carrier10 \\ \hline
Shipper1 & 1.07 & 1.04 & 1.10 & 1.25 & 1.13 & 1.07 & 1.04 & 1.10 & 1.25 & 1.13 \\
Shipper2 & 1.08 & 1.14 & 0.99 & 1.41 & 1.25 & 1.08 & 1.14 & 0.99 & 1.41 & 1.25 \\
Shipper3 & 0.88 & 0.99 & 1.13 & 0.97 & 1.15 & 0.88 & 0.99 & 1.13 & 0.97 & 1.15 \\
Shipper4 & 1.30 & 1.36 & 1.12 & 1.20 & 0.72 & 1.30 & 1.36 & 1.12 & 1.20 & 0.72 \\
Shipper5 & 1.07 & 1.04 & 1.10 & 1.25 & 1.13 & 1.07 & 1.04 & 1.10 & 1.25 & 1.13 \\
Shipper6 & 1.08 & 1.14 & 0.99 & 1.41 & 1.25 & 1.08 & 1.14 & 0.99 & 1.41 & 1.25 \\
Shipper7 & 0.88 & 0.99 & 1.13 & 0.97 & 1.15 & 0.88 & 0.99 & 1.13 & 0.97 & 1.15 \\
Shipper8 & 1.30 & 1.36 & 1.12 & 1.20 & 0.72 & 1.30 & 1.36 & 1.12 & 1.20 & 0.72
\\
\hline 
\end{tabular}}
\end{table}

\begin{table}[!htbp]  
\caption{Carrier's reliability indicator matrix for evaluating shippers}\label{Carrierreliabilitytab}
\resizebox{\textwidth}{.80in}{
\begin{tabular}{ccccccccc}
\hline 
          & Shipper1 & Shipper2 & Shipper3 & Shipper4 & Shipper5 & Shipper6 & Shipper7 & Shipper8 \\ \hline 
Carrier1  & 1.16     & 1.39     & 0.91     & 1.40     & 1.16     & 1.39     & 0.91     & 1.40     \\
Carrier2  & 0.73     & 0.99     & 0.76     & 1.40     & 0.73     & 0.99     & 0.76     & 1.40     \\
Carrier3  & 0.94     & 0.92     & 0.81     & 1.15     & 0.94     & 0.92     & 0.81     & 1.15     \\
Carrier4  & 1.20     & 0.85     & 0.75     & 0.86     & 1.20     & 0.85     & 0.75     & 0.86     \\
Carrier5  & 0.73     & 1.61     & 1.11     & 1.08     & 0.73     & 1.61     & 1.11     & 1.08     \\
Carrier6  & 1.16     & 1.39     & 0.91     & 1.40     & 1.16     & 1.39     & 0.91     & 1.40     \\
Carrier7  & 0.73     & 0.99     & 0.76     & 1.40     & 0.73     & 0.99     & 0.76     & 1.40     \\
Carrier8  & 0.94     & 0.92     & 0.81     & 1.15     & 0.94     & 0.92     & 0.81     & 1.15     \\
Carrier9  & 1.20     & 0.85     & 0.75     & 0.86     & 1.20     & 0.85     & 0.75     & 0.86     \\
Carrier10 & 0.73     & 1.61     & 1.11     & 1.08     & 0.73     & 1.61     & 1.11     & 1.08    \\
\hline 
\end{tabular} }
\end{table}

Through personnel surveys and expert interviews, the indicator weights are obtained by the RSDAT indicator weighting method (see Table \ref{weighttab}), since the attribute indicator does not participate in the comprehensive satisfaction aggregation, its weights are all equal to 0.

\begin{table}[!htbp]
\centering
\caption{The weights of the indicators considered by both matching sides}\label{weighttab}
\resizebox{.8\textwidth}{.25in}{
\begin{tabular}{cccccccccc}
\hline
 &
  \multicolumn{1}{l}{$w(A^{1})$} &
  \multicolumn{1}{l}{$w(A^{2})$} &
  \multicolumn{1}{l}{$w(A^{3})$} &
  \multicolumn{1}{l}{$w(A^{4})$} &
  \multicolumn{1}{l}{$w(A^{5})$} &
  \multicolumn{1}{l}{$w(A^{6})$} &
  \multicolumn{1}{l}{$w(A^{7})$} &
  \multicolumn{1}{l}{$w(A^{8})$} &
  \multicolumn{1}{l}{$w(A^{9})$} \\
  \hline
Shipper-side &
  0.28 &
  0.26 &
  0.18 &
  0 &
  0 &
  0 &
  0 &
  0 &
  0.28 \\
Carrier-side &
  0 &
  0.43 &
  0 &
  0 &
  0 &
  0.24 &
  0 &
  0 &
  0.33\\
\hline
\end{tabular}}
\end{table}

The comprehensive satisfaction among vehicle-cargo matching subjects is obtained (see Table \ref{shippersatisfactiontab} - Table \ref{carriersatisfactiontab}).

\begin{table}[!htbp]
\centering
\caption{Shipper's satisfaction to carrier}\label{shippersatisfactiontab}
\resizebox{\textwidth}{.7in}{
\begin{tabular}{ccccccccccc}
\hline 
 & Carrier1 & Carrier2 & Carrier3 & Carrier4 & Carrier5 & Carrier6 & Carrier7 & Carrier8 & Carrier9 & Carrier10 \\
 \hline 
Shipper1 & 0.60 & 0.51 & 0.62 & 0.94 & 0.85 & 0.60 & 0.51 & 0.62 & 0.94 & 0.85 \\
Shipper2 & 0.96 & 0.91 & 0.61 & 0.81 & 0.74 & 0.96 & 0.91 & 0.61 & 0.81 & 0.74 \\
Shipper3 & 0.83 & 0.67 & 0.98 & 0.54 & 0.58 & 0.83 & 0.67 & 0.98 & 0.54 & 0.58 \\
Shipper4 & 0.81 & 0.73 & 0.92 & 0.48 & 0.31 & 0.81 & 0.73 & 0.92 & 0.48 & 0.31 \\
Shipper5 & 0.60 & 0.51 & 0.62 & 0.94 & 0.85 & 0.60 & 0.51 & 0.62 & 0.94 & 0.85 \\
Shipper6 & 0.96 & 0.91 & 0.61 & 0.81 & 0.74 & 0.96 & 0.91 & 0.61 & 0.81 & 0.74 \\
Shipper7 & 0.83 & 0.67 & 0.98 & 0.54 & 0.58 & 0.83 & 0.67 & 0.98 & 0.54 & 0.58 \\
Shipper8 & 0.81 & 0.73 & 0.92 & 0.48 & 0.31 & 0.81 & 0.73 & 0.92 & 0.48 & 0.31\\ \hline
\end{tabular} }
\end{table}

\begin{table}[!htbp]
\centering
\caption{Carrier's satisfaction to shipper}\label{carriersatisfactiontab}
\resizebox{\textwidth}{.8in}{
\begin{tabular}{ccccccccc}
\hline 
          & Shipper1 & Shipper2 & Shipper3 & Shipper4 & Shipper5 & Shipper6 & Shipper7 & Shipper8 \\
          \hline
Carrier1  & 0.74     & 0.93     & 0.87     & 0.84     & 0.74     & 0.93     & 0.87     & 0.84     \\
Carrier2  & 0.62     & 0.88     & 0.82     & 0.92     & 0.62     & 0.88     & 0.82     & 0.92     \\
Carrier3  & 0.60     & 0.72     & 0.82     & 0.84     & 0.60     & 0.72     & 0.82     & 0.84     \\
Carrier4  & 0.78     & 0.75     & 0.80     & 0.82     & 0.78     & 0.75     & 0.80     & 0.82     \\
Carrier5  & 0.66     & 0.98     & 0.97     & 0.87     & 0.66     & 0.98     & 0.97     & 0.87     \\
Carrier6  & 0.74     & 0.93     & 0.87     & 0.84     & 0.74     & 0.93     & 0.87     & 0.84     \\
Carrier7  & 0.62     & 0.88     & 0.82     & 0.92     & 0.62     & 0.88     & 0.82     & 0.92     \\
Carrier8  & 0.60     & 0.72     & 0.82     & 0.84     & 0.60     & 0.72     & 0.82     & 0.84     \\
Carrier9  & 0.78     & 0.75     & 0.80     & 0.82     & 0.78     & 0.75     & 0.80     & 0.82     \\
Carrier10 & 0.66     & 0.98     & 0.97     & 0.87     & 0.66     & 0.98     & 0.97     & 0.87     \\ \hline 
\end{tabular} } 
\end{table}

According to Eqs. (\ref{29} - \ref{35}), we can get $f_{1}^L=3.94,f_{1}^U=7.28,f_{2}^L=5.96,f_{2}^U=7.2,f_{1}^{UL}=6.61,f_{2}^{UL}=6.95$ by programming CPLEX, and the satisfactory matching scheme is (1,4),(2,1),(3,10),(4,7),(5,9),(6,6),(7,3),(8,2). Through substituting the obtained satisfactory matching scheme into Eq. (\ref{33}), we find that shipper side's satisfaction is 0.86 and carrier side's satisfaction is 0.88, the overall satisfaction is 1.74, $\eta=u_{1} \left ( f_{1} \left ( x \right )  \right ) / u_{2} \left ( f_{2} \left ( x \right )  \right )=0.98 $, which satisfies the fairness measurement standard.

\subsubsection{Sensitivity analysis}
By adjusting the value of parameter $\gamma$ (i.e., changing the two matching sides' dissatisfaction interval limitation), we can get different optimal matching schemes obtained from $LP3$, which affects the fairness factor directly (see Table \ref{Algorithm2}). To elucidate the improvement model's satisfactory vehicle-cargo matching scheme, we will provide further explanations.
\begin{table}[!htbp] 
\caption{Algorithm2}\label{Algorithm2}
\begin{tabular}{|cl|}
\hline
Algorithm 2: & Adjusting the value of parameter $\gamma $ to test the output of matching scheme
\\ \hline
Input:       & $\left [ n_{L},n_{H} \right ]$                                                                                                                         \\ 
Output:      & Matching scheme's sensitivity analysis                                                                                                                                       \\ 
            & Begin                                                                                                                                                                \\ 
Step1:       &  Set the satisfaction measurement interval of the matching platform$\left[  n_{L},n_{H}\right] $ \\ 
Step2:       &  For $\gamma$ =0.1:0.1:1  do                                 \\ 
           & \quad Run step3 $\sim $ step5 of algorithm 1                      \\ 
          & End for                                                  \\ 
          & End                                               \\  
\hline
\end{tabular}
\end{table}

For the case in this paper, sensitivity analysis is carried out from the perspective of the shipper side and the carrier side respectively, and the value of parameter $\gamma$  changes according to the construction of the non-membership degree function (step size is 0.1, from 0.1 to 1). Through programming CPLEX, the output results can be obtained (see Figure \ref{Figsensitivity} and Table \ref{Tabsensitivity}):

\begin{figure}[H]
\centering
\includegraphics[width=\textwidth]{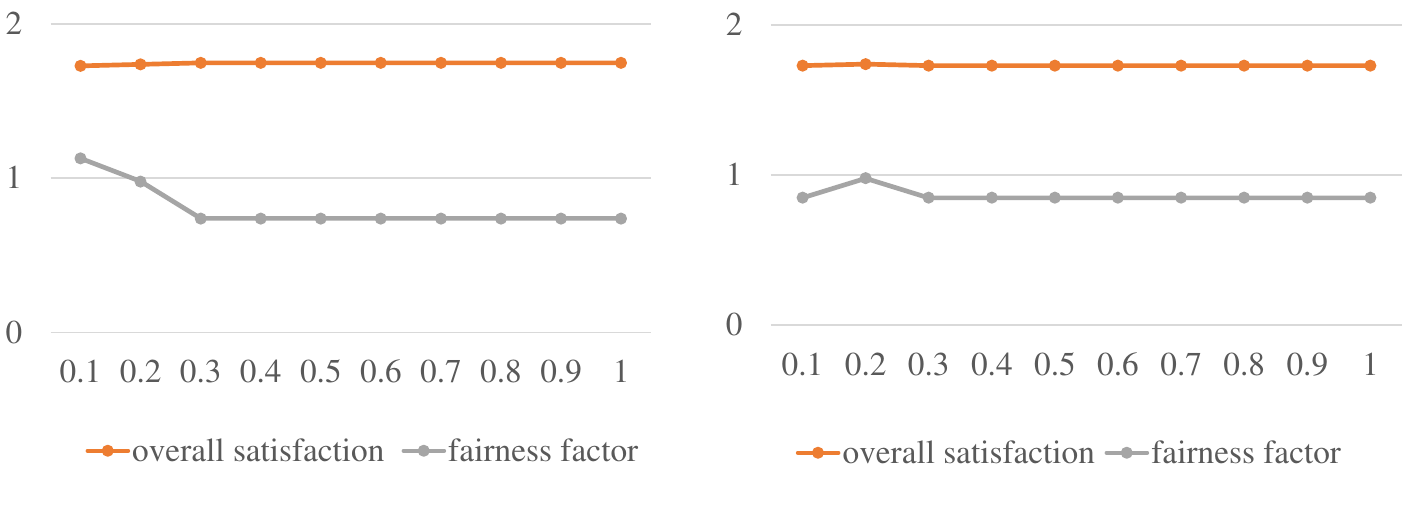}
\caption{Shipper-side and carrier-side sensitivity analysis}
\label{Figsensitivity}
\end{figure}

\begin{table}[H]
\centering
\caption{Sensitivity  analysis result}\label{Tabsensitivity}
\begin{tabular}{ccccccccccc}
\hline 
 &
  \multicolumn{10}{c}{Parameter} \\
  \hline 
 &
  0.1 &
  $\eta$ &
  0.2 &
  $\eta$ &
  0.3 &
  $\eta$ &
  0.4 &
  $\eta$ &
  0.5 &
  $\eta$ \\
  \hline 
Shipper-side &
  6.946 &
  \multirow{2}{*}{1.13} &
  6.612 &
  \multirow{2}{*}{0.98} &
  6.278 &
  \multirow{2}{*}{0.74} &
  5.944 &
  \multirow{2}{*}{0.74} &
  5.610 &
  \multirow{2}{*}{0.74} \\
Overall satisfaction &
  1.73 &
   &
  1.74 &
   &
  1.75 &
   &
  1.75 &
   &
  1.75 &
   \\
   \hline
Matching scheme &
  \multicolumn{2}{l}{\begin{tabular}[c]{@{}l@{}}(1,4)(2,1)\\ (3,5)(4,8)\\ (5,9)(6,6)\\ (7,3)(8,2)\end{tabular}} &
  \multicolumn{2}{l}{\begin{tabular}[c]{@{}l@{}}(1,4)(2,1)\\ (3,5)(4,8)\\ (5,9)(6,6)\\ (7,3)(8,2)\end{tabular}} &
  \multicolumn{2}{l}{\begin{tabular}[c]{@{}l@{}}(1,4)(2,1)\\ (3,5)(4,8)\\ (5,9)(6,6)\\ (7,3)(8,2)\end{tabular}} &
  \multicolumn{2}{l}{\begin{tabular}[c]{@{}l@{}}(1,4)(2,1)\\ (3,5)(4,8)\\ (5,9)(6,6)\\ (7,3)(8,2)\end{tabular}} &
  \multicolumn{2}{l}{\begin{tabular}[c]{@{}l@{}}(1,4)(2,1)\\ (3,5)(4,8)\\ (5,9)(6,6)\\ (7,3)(8,2)\end{tabular}} \\
  \hline
Carrier-side &
  7.076 &
  \multirow{2}{*}{0.85} &
  6.952 &
  \multirow{2}{*}{0.98} &
  6.828 &
  \multirow{2}{*}{0.85} &
  6.704 &
  \multirow{2}{*}{0.85} &
  6.580 &
  \multirow{2}{*}{0.85} \\
 
Overall satisfaction &
  1.73 &
   &
  1.74 &
   &
  1.73 &
   &
  1.73 &
   &
  1.73 &
   \\
   \hline
Matching scheme &
  \multicolumn{2}{l}{\begin{tabular}[c]{@{}l@{}}(1,4)(2,1)\\ (3,5)(4,3)\\ (5,9)(6,6)\\ (7,10)(8,2)\end{tabular}} &
  \multicolumn{2}{l}{\begin{tabular}[c]{@{}l@{}}(1,4)(2,1)\\ (3,10)(4,7)\\ (5,9)(6,6)\\ (7,3)(8,2)\end{tabular}} &
  \multicolumn{2}{l}{\begin{tabular}[c]{@{}l@{}}(1,9)(2,6)\\ (3,5)(4,8)\\ (5,9)(6,1)\\ (7,10)(8,7)\end{tabular}} &
  \multicolumn{2}{l}{\begin{tabular}[c]{@{}l@{}}(1,9)(2,6)\\ (3,5)(4,8)\\ (5,4)(6,1)\\ (7,10)(8,7)\end{tabular}} &
  \multicolumn{2}{l}{\begin{tabular}[c]{@{}l@{}}(1,9)(2,6)\\ (3,5)(4,8)\\ (5,4)(6,1)\\ (7,10)(8,7)\end{tabular}} \\
  \hline
 &
  0.6 &
  $\eta$ &
  0.7 &
  $\eta$ &
  0.8 &
  $\eta$ &
  0.9 &
  $\eta$ &
  1 &
  $\eta$ \\
  \hline

Shipper-side &
  5.276 &
  \multirow{2}{*}{0.74} &
  4.942 &
  \multirow{2}{*}{0.74} &
  4.608 &
  \multirow{2}{*}{0.74} &
  4.274 &
  \multirow{2}{*}{0.74} &
  3.940 &
  \multirow{2}{*}{0.74} \\

Overall satisfaction &
  1.75 &
   &
  1.75 &
   &
  1.75 &
   &
  1.75 &
   &
  1.75 &
   \\
   \hline
Matching scheme &
  \multicolumn{2}{l}{\begin{tabular}[c]{@{}l@{}}(1,4)(2,1)\\ (3,10)(4,7)\\ (5,9)(6,6)\\ (7,5)(8,2)\end{tabular}} &
  \multicolumn{2}{l}{\begin{tabular}[c]{@{}l@{}}(1,4)(2,1)\\ (3,10)(4,7)\\ (5,9)(6,6)\\ (7,5)(8,2)\end{tabular}} &
  \multicolumn{2}{l}{\begin{tabular}[c]{@{}l@{}}(1,4)(2,1)\\ (3,10)(4,7)\\ (5,9)(6,6)\\ (7,5)(8,2)\end{tabular}} &
  \multicolumn{2}{l}{\begin{tabular}[c]{@{}l@{}}(1,4)(2,1)\\ (3,10)(4,7)\\ (5,9)(6,6)\\ (7,5)(8,2)\end{tabular}} &
  \multicolumn{2}{l}{\begin{tabular}[c]{@{}l@{}}(1,4)(2,1)\\ (3,10)(4,7)\\ (5,9)(6,6)\\ (7,5)(8,2)\end{tabular}} \\
  \hline
Carrier-side &
  6.456 &
  \multirow{2}{*}{0.85} &
  6.332 &
  \multirow{2}{*}{0.85} &
  6.208 &
  \multirow{2}{*}{0.85} &
  6.084 &
  \multirow{2}{*}{0.85} &
  5.960 &
  \multirow{2}{*}{0.85} \\
  
Overall satisfaction &
  1.73 &
   &
  1.73 &
   &
  1.73 &
   &
  1.73 &
   &
  1.73 &
   \\
   \hline
Matching scheme &
  \multicolumn{2}{l}{\begin{tabular}[c]{@{}l@{}}(1,9)(2,6)\\ (3,5)(4,8)\\ (5,4)(6,1)\\ (7,10)(8,7)\end{tabular}} &
  \multicolumn{2}{l}{\begin{tabular}[c]{@{}l@{}}(1,9)(2,6)\\ (3,5)(4,8)\\ (5,4)(6,1)\\ (7,10)(8,7)\end{tabular}} &
  \multicolumn{2}{l}{\begin{tabular}[c]{@{}l@{}}(1,9)(2,6)\\ (3,5)(4,8)\\ (5,4)(6,1)\\ (7,10)(8,7)\end{tabular}} &
  \multicolumn{2}{l}{\begin{tabular}[c]{@{}l@{}}(1,9)(2,6)\\ (3,5)(4,8)\\ (5,4)(6,1)\\ (7,10)(8,7)\end{tabular}} &
  \multicolumn{2}{l}{\begin{tabular}[c]{@{}l@{}}(1,9)(2,6)\\ (3,5)(4,8)\\ (5,4)(6,1)\\ (7,10)(8,7)\end{tabular}}
  \\
  \hline
\end{tabular}
\end{table}

Because $f_{1}^{UL}=f_{1}^{U}-\gamma \left( f_{1}^U- f_{1}^L \right)$,$f_{2}^{UL}=f_{2}^{U}-\gamma \left( f_{2}^U- f_{2}^L \right)$, the change of the parameter $\gamma$ will directly affect the non-membership degree function $v_{1}\left( f_{1}(x) \right) $ and $v_{2}\left( f_{2}(x) \right) $ (which  equals to change the slope of  Figure \ref{Matching objective function}, see Figure \ref{the objective function non-membership function}), then the optimal matching scheme of will be affected. Besides, the overall satisfaction $s_{h} \left ( f_{h} \left ( x \right )  \right ) =u_{h} \left ( f_{h} \left ( x \right )  \right )-v_{h} \left ( f_{h} \left ( x \right )  \right ), h=1,2$ and the value of the fairness factor $\eta =s_{1}(f_{1}(x))/s_{2}(f_{2}(x))$  will be changed continuously.

\begin{figure}[!htbp]
\centering
\includegraphics[width=.6\textwidth]{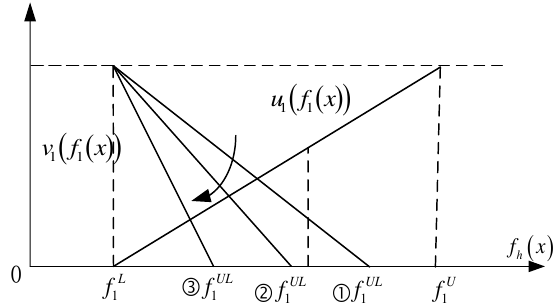}
\caption{$\gamma$ affects the objective function non-membership function (take $f_{1}$ as an example)}
\label{the objective function non-membership function}
\end{figure}

In the first case, the changed value is $f_{1}^{UL}$, and the unchanged value is $f_{2}^{UL}$. When $\gamma$ = 0.1, the model outputs the optimal matching scheme (1,4), (2,1), (3,5), (4,8), (5,9), (6,6), (7,3), (8,2) obtained from $LP3$. At this time, the overall satisfaction is 1.73 and the fairness factor is 1.13, which does not meet the fairness factor's pre-set value interval. When $\gamma$ =0.2, it is the satisfactory matching scheme. When $\gamma$ changes from 0.3 to 1 (step size 0.1), the model outputs the same matching scheme (1,4), (2,1), (3,10), (4,7), (5,9), (6,6), (7,5), (8,2) obtained from $LP3$, the overall satisfaction is 1.75, the fairness factor is 0.74, which does not meet the fairness factor's pre-set value interval. 

In the second case, the changed value becomes $f_{2}^{UL}$, while the unchanged value is $f_{1}^{UL}$  .When $\gamma$ = 0.1, the model outputs the optimal matching scheme $(1,4),(2,1),(3,5),(4,3),(5,9)$,  $(6,6),(7,10),(8,2)$ obtained from $LP3$. At this time, the overall satisfaction is 1.73 and the fairness factor is 0.85, which does not meet the fairness factor's pre-set value interval. When $\gamma$  = 0.2, it is the satisfactory matching scheme. When $\gamma$ changes from 0.3 to 1 (step size 0.1), the model outputs the same matching scheme (1,9), (2,6), (3,5), (4,8), (5,4), (6,1), (7,10), (8,7) obtained from $LP3$, the overall satisfaction is 1.73, the fairness factor is 0.85, which does not meet the fairness factor's pre-set value interval.

Through the sensitivity analysis, it can be seen that the larger the value of $\gamma$, the smaller the decision maker's unsatisfactory interval. Hence, if the fairness factor does not meet the decision-maker's pre-set value interval, it can be adjusted by changing the value of $\gamma$  or adding fairness adjustment constraints. The overall satisfaction of the optimal matching scheme obtained from $LP3$ will change, and the fairness factor will also change accordingly.

\subsubsection{Comparative analysis}
In order to demonstrate the effectiveness of the proposed solution method, we conduct a comparative analysis with three benchmark methods for solving dual-objective matching models, namely, the max-min method, the ideal point method, and the linear weighting method (see Table \ref{Comparetab}). 

\begin{table}[!h]
\centering
\caption{Satisfactory matching scheme obtained by different solving methods}\label{Comparetab}
\resizebox{.9\textwidth}{.9in}{
\begin{tabular}{ccccccccc}
\hline 
\multirow{2}{*}{Methods} &
  \multirow{2}{*}{Matching scheme} &
  \multicolumn{3}{c}{\begin{tabular}[c]{@{}c@{}}$LP1$\\ (Before improvement)\end{tabular}} &
  \multicolumn{4}{c}{\begin{tabular}[c]{@{}c@{}}$LP3$\\ (After improvement)\end{tabular}} \\

 &
   &
  $f_{1}$ &
  $f_{2}$ &
  $f_{1}+f_{2}$ &
  $s_{1}$ &
  $s_{2}$ &
  $s_{1}+s_{2}$ &
  $\eta$ \\ 
    \hline 
Max\_min &
  \begin{tabular}[c]{@{}c@{}}(1,4)(2,1)(3,8)(4,7)\\ (5,9)(6,6)(7,3)(8,2)\end{tabular} &
  7.22 &
  6.9 &
  14.12 &
  0.98 &
  0.71 &
  1.69 &
  1.39 \\
Ideal point &
  \begin{tabular}[c]{@{}c@{}}(1,4)(2,1)(3,8)(4,7)\\ (5,9)(6,6)(7,3)(8,2)\end{tabular} &
  7.22 &
  6.9 &
  14.12 &
  0.98 &
  0.71 &
  1.69 &
  1.39 \\
Linear weighted &
  \begin{tabular}[c]{@{}c@{}}(1,4)(2,1)(3,8)(4,7)\\ (5,9)(6,6)(7,3)(8,2)\end{tabular} &
  7.22 &
  6.9 &
  14.12 &
  0.98 &
  0.71 &
  1.69 &
  1.39 \\
Proposed method &
  \begin{tabular}[c]{@{}c@{}}(1,4)(2,1)(3,10)(4,7)\\ (5,9)(6,6)(7,3)(8,2)\end{tabular} &
  6.82 &
  7.05 &
  13.87 &
  0.86 &
  0.88 &
  1.74 &
  0.98\\
\hline 
\end{tabular} }
\end{table}

As Table \ref{Comparetab} shows, the three benchmark methods yield the same matching scheme, which is different from that of the proposed method. Referring to $LP1$ (before improvement), the overall satisfaction ($f_{1}+f_{2}$) of the satisfactory matching scheme obtained by the method proposed in this paper is lower than the satisfactory matching scheme obtained by the other three methods; however, referring to $LP3$ (after improvement), its overall satisfaction ($s_{1}+s_{2} $) is the highest, and the fairness factor meets the pre-set value interval. 

The main reason for this difference is: that the matching scheme obtained by the other traditional methods makes the overall satisfaction ($f_{1}+f_{2}$) of $LP1$  (before improvement) reach the maximum level; however, $LP3$ (after improvement) has unsatisfactory interval limitation to integrate the matching objective function's non-membership degree information, the value of $f_{2}$ is less than $f_{2}^{UL}$ (i.e.,  $f_{2}=6.9<f_{2}^{UL}=6.95$) compared with the satisfactory matching scheme obtained by the proposed method, the matching scheme causes a decrease in the value of $u_{2}\left( f_{2} \left( x \right) \right)$  and an increase in the value of  $v_{2}\left( f_{2} \left( x \right) \right)$, resulting in a decrease in the satisfaction of $s_{2}$. The decrease in satisfaction of $s_{2}$ (-0.17) is greater than the increase in satisfaction of $s_{1}$ (+0.12), which eventually leads to a decrease in overall satisfaction (-0.05), and its fairness factor does not meet the pre-set value interval, so it is excluded from the solution process of $LP3$  .

Based on the intuitionistic fuzzy set theory, our improved model fully incorporates the membership degree, non-membership degree, and hesitation degree information of the two matching objectives. Additionally, the model considers the fairness factor to obtain a satisfactory matching scheme that meets the pre-set value interval. In contrast, traditional methods neglect this information, so its satisfactory matching scheme decreases the overall satisfaction ($s_{1}+s_{2}$) of $LP3$ (after improvement), and the fairness factor cannot meet the pre-set value interval. The above analysis shows the effectiveness of the method proposed in this paper.

\subsection{Simulation analysis}\label{7.0}

The evolutionary game model centers on the strategic selection of stakeholders, with a primary emphasis on uncovering the regularities governing the dynamic system's evolution. Consequently, ensuring the accuracy of the dynamic system's structure takes precedence over selecting parameter values \citep{CGL2018,ZFL2020}. To conduct a comprehensive and systematic examination of stakeholder strategic selection, we have devised a simulation analysis. The choice of all simulation parameters' values in this paper does not represent all parties' payment or profit values in reality. Instead, they primarily serve to explore the impact of variations in relevant parameter values on the strategic selections made by stakeholders.
\subsubsection{Data and scenarios}
According to “The 2022 “Sunshine Action” Work Plan on Building a New Transport Business Platform for Enterprises' proportion of commission” issued by the Ministry of Transport of the People's Republic of China, this paper sets the percentage of commission to 10\%, which is no more than its upper limit (18\%-30\%).

Based on the basic assumptions in Section \ref{6.1}, we have:
\begin{equation}
q_{1}+q_{2}=1
\end{equation}
\begin{equation}
p_{1}+p_{2}=1
\end{equation}
                                                                           
$Sub_{I}$ and $Sub_{P}$ represent the platform's subsidy amount when shipper and carrier select “acceptance” strategy, we have:
\begin{equation}
Sub_{I}=\alpha S_{I}
\end{equation}
\begin{equation}
Sub_{P}=\alpha S_{P}
\end{equation}
                                                                                 
However, some exogenous variables are difficult to quantify in the actual environment directly (i.e., $D_{G}$ , $F_{G}^I$ and $F_{G}^P$  ), which will be determined by reasonable assumptions based on Table \ref{incometable}. In detail, when the dynamic system has a unique stable equilibrium point (1,1,0), there is array 1:  $C_{I}=5,\sigma _{1} =1,h=0.1,W_{I}=20,Q_{I}=10,q_{2}=0.6,R_{I}^L=25,q_{1}=0.4,R_{I}^U=35,R_{I}=30,u_{I}=0.7,C_{p}=5,\sigma _{2} =0.2,\alpha=0.6,S_{I}=10,f=0.1, W_{p}=20,Q_{p}=10,p_{2}=0.6,R_{P}^L=25,p_{1}=0.4,R_{p}^U=35,R_{p}=30,u_{p}=0.7, \beta=0.6,S_{p}=10,D_{G}=3,\eta=0.7,F_{G}^I=10,F_{G}^P=10$.When the dynamic system has a unique stable equilibrium point (1,1,1), there is an array 2:$D_{G}=0.5$, and the rest of the parameters are the same as array 1.

\subsubsection{Model checking}
In order to test the validity of the dynamic system evolution stability analysis in the paper, the parameters of array 1 ($E_{7}$) and array 2 ($E_{8}$) were substituted into the model for MATLAB simulation (The change step size of $x$, $y$, $z$ is 0.2:0.2:0.8). The simulation results are shown in Figure \ref{Fig11Fig12}:

\begin{figure}[htbp]
\centering
\includegraphics[width=.70\textwidth]{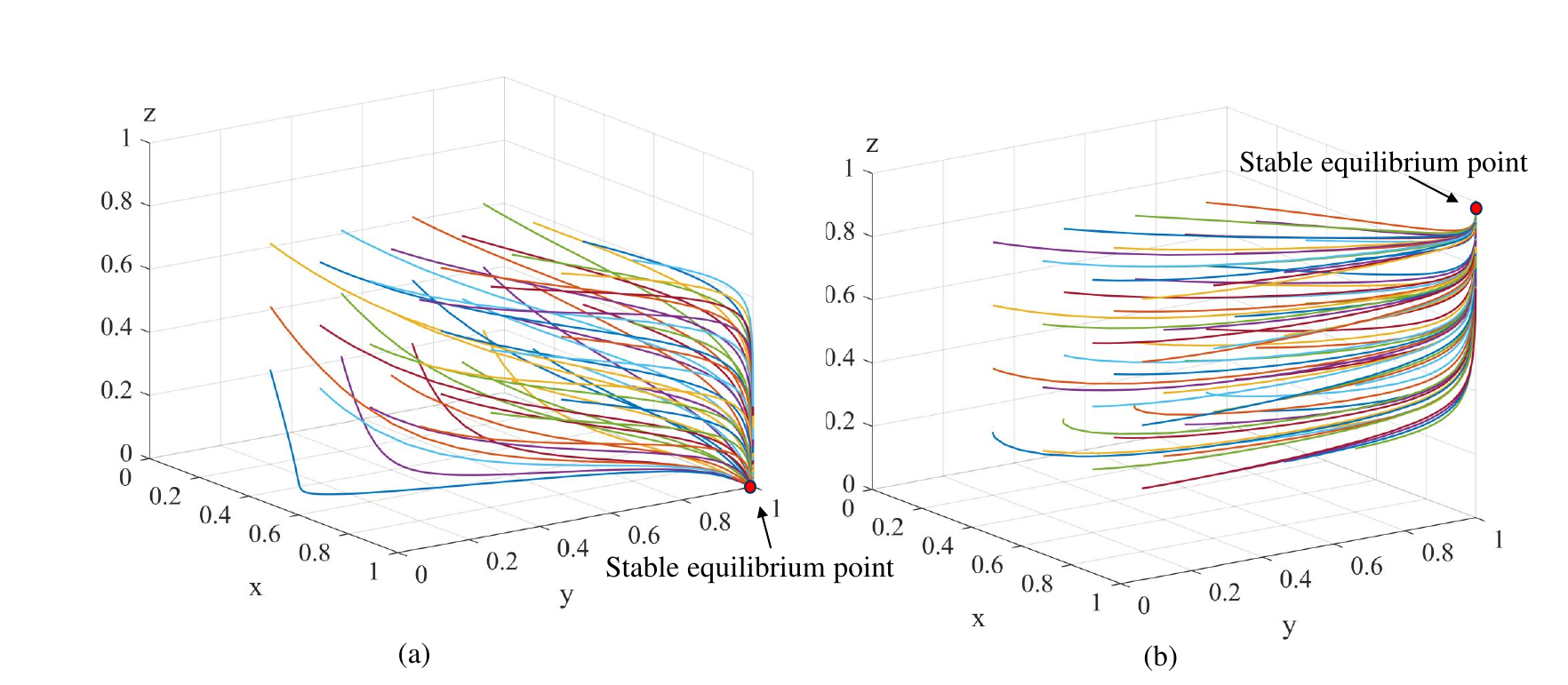}
\caption{System simulation diagram of array 1 and array 2}
\label{Fig11Fig12}
\end{figure}

It can be seen from these figures that under the condition of array 1 (see Figure \ref{Fig11Fig12}a), the system eventually evolves to (1,1,0), which is consistent with the conclusion of $E_{7}$. Under array 2 (see Figure \ref{Fig11Fig12}b), the system will eventually evolve to (1,1,1), consistent with $E_{8}$. The simulation analysis results are consistent with the stability analysis of system evolution. The tripartite evolutionary game model is effective and has important practical guiding significance. In order to eliminate the influence of the initial probability value setting, the initial probability value of x, y, z is all fixed at 0.6, and the influence of platform subsidy intensity, service level, and fairness factor on the system evolution results will be discussed respectively.
\subsubsection{Platform subsidy intensity $\alpha$} 
In order to study the influence of changing platform subsidy intensity $\alpha$ on the system evolution results under $E_{7}$  and $E_{8} $, the other parameters of array 1 and array 2 are kept unchanged, and the value of $\alpha$ is changed by 0.2:0.2:0.8. The simulation results are shown in Figure \ref{Fig13Fig14}:

\begin{figure}[h]  
\centering
\includegraphics[width=.70\textwidth]{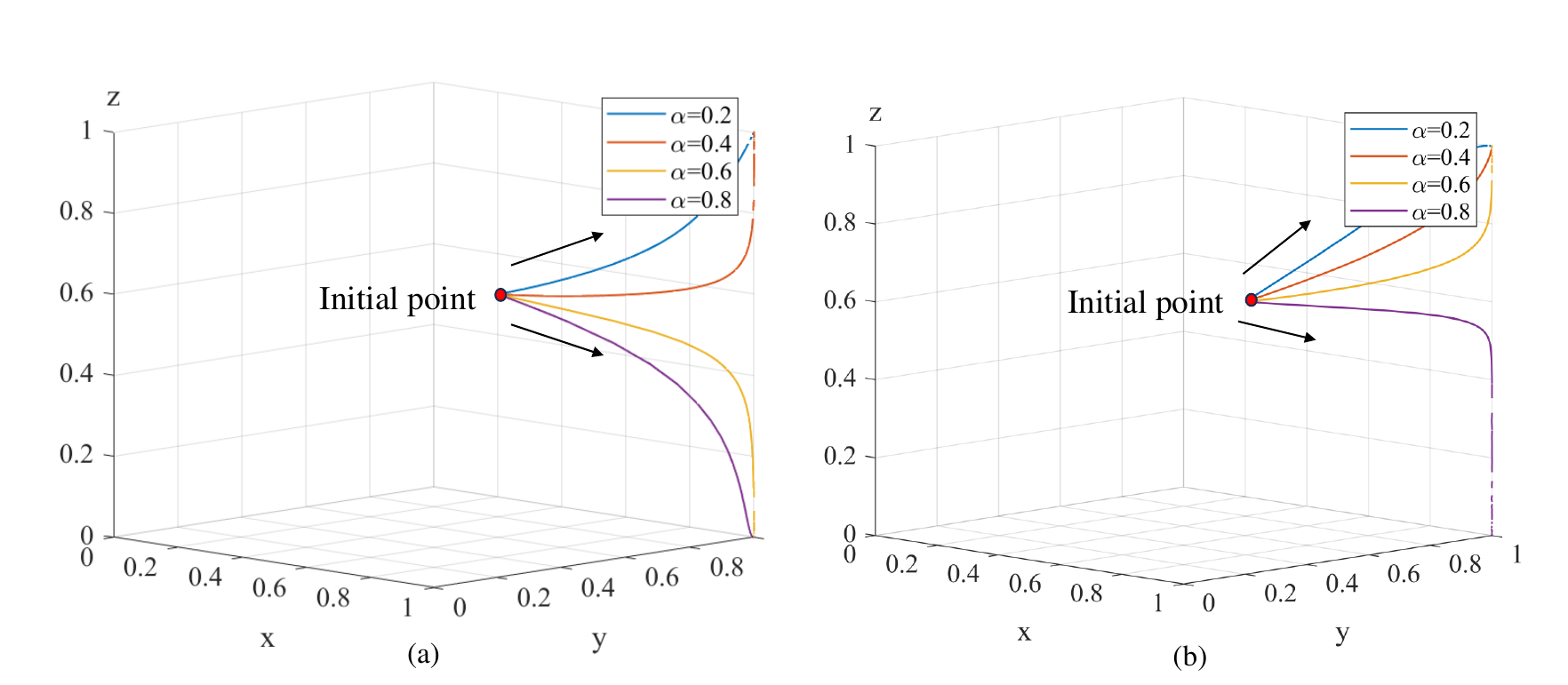}
\caption{$\alpha$ evolutionary game simulation of array 1 and array 2}
\label{Fig13Fig14}
\end{figure}

It can be seen from Figure \ref{Fig13Fig14}a that, under array 1, the critical value of $\alpha$   is between 0.4 and 0.6. When $\alpha$  is less than this critical value, the shipper, carrier, and the platform finally select acceptance, acceptance, and subsidy respectively, the dynamic system evolutionary stable point is (1,1,1); when   $\alpha$ exceeds this critical value, shipper and carrier still choose acceptance, but the platform will choose no subsidy, resulting in the dynamic system evolutionary stable point of (1,1,0). According to the change rate of the evolution curve, it can also be seen that when the value of $\alpha$ increases from 0.6, the rate at which the dynamic system evolves to (1,1,0) speeds up. This is because under array 1, even if the platform does not subsidize, the shipper and the carrier will also gain considerable benefits by accepting the matching scheme recommended by the platform, that is, no subsidy doesn't influence the carrier's and shipper's strategic selection; when the value of $\alpha$ starts to decrease from 0.4, the rate which dynamic system evolves to (1,1,1) speeds up, it can be seen that the appropriate subsidy intensity can not only save the subsidy cost but also accelerate the rate which dynamic system evolves to (1,1,1) stable state. Compared with the condition of array 1, the critical value of $\alpha$ is between 0.6 and 0.8 in the condition of array 2 (see Figure \ref{Fig13Fig14}b), the rest are the same which will not be repeated here.

\subsubsection{ Platform subsidy intensity $\beta$}
In order to study the influence of changing platform subsidy intensity $\beta $ on the system evolution results under $ E_{7} $ and  $ E_{8} $, the other parameters of arrays 1 and array 2 are kept unchanged, and the value of $\beta $ is changed by 0.2:0.2:0.8. The simulation results are shown in Figure \ref{Fig15Fig16}:

\begin{figure}[h!]
\centering
\includegraphics[width=.70\textwidth]{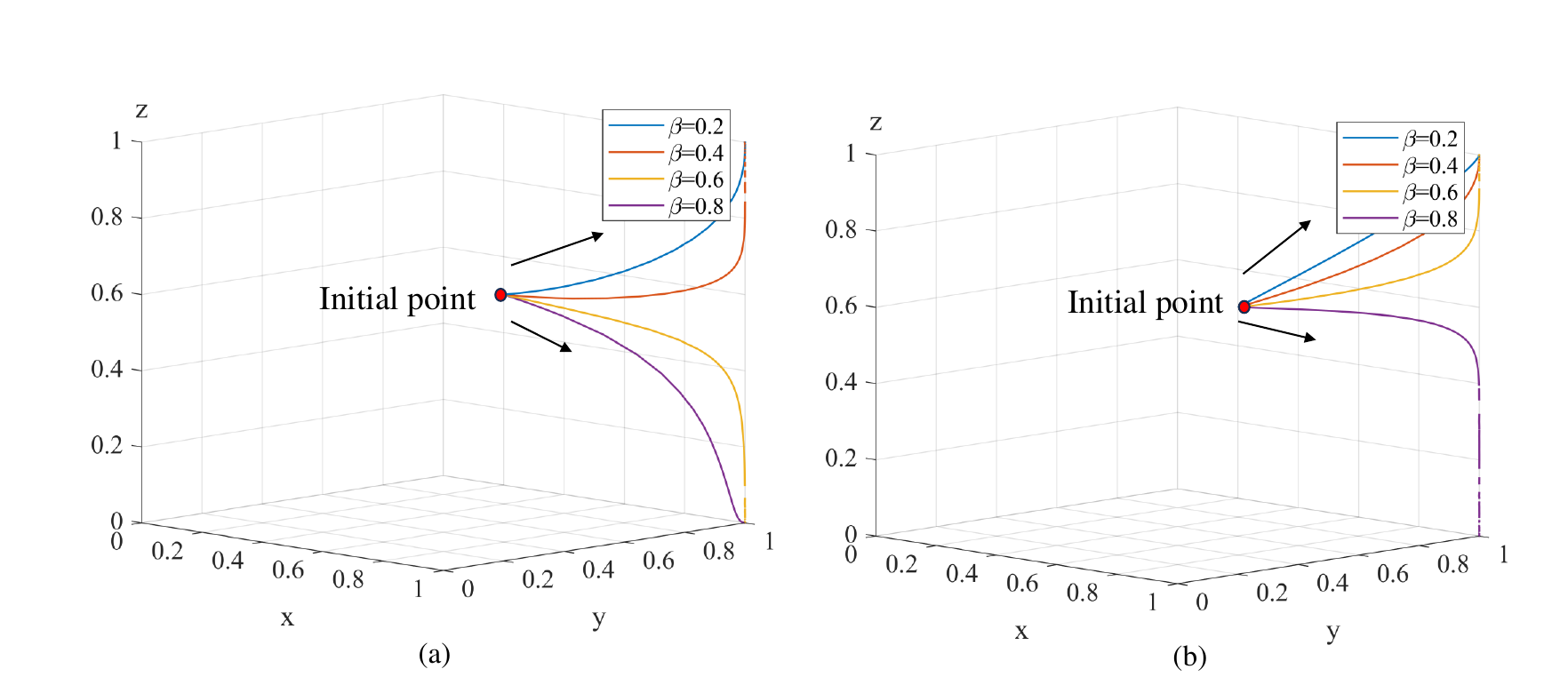}
\caption{$\beta$ evolutionary game simulation of array 1 and array 2}
\label{Fig15Fig16}
\end{figure}

According to these figures: under array 1 (see Figure \ref{Fig15Fig16}a) and array 2 (see Figure \ref{Fig15Fig16}b), the influence of platform subsidy intensity $\beta$  on system evolution results is the same as the platform subsidy intensity $\alpha$ 's, which will not be repeated here.

\subsubsection{Shipper service level  $u_{I}$ (referring to the service level enjoyed by the shipper)}
In order to study the influence of changing shipper service level $u_{I}$ on the system evolution results under $E_{7}$  and $E_{8}$, the other parameters of array 1 and array 2 are kept unchanged, and the value of $u_{I}$ is changed by 0.2:0.2:0.8.The simulation results are shown in Figure \ref{Fig17Fig18}:

\begin{figure}[h!]
\centering
\includegraphics[width=.70\textwidth]{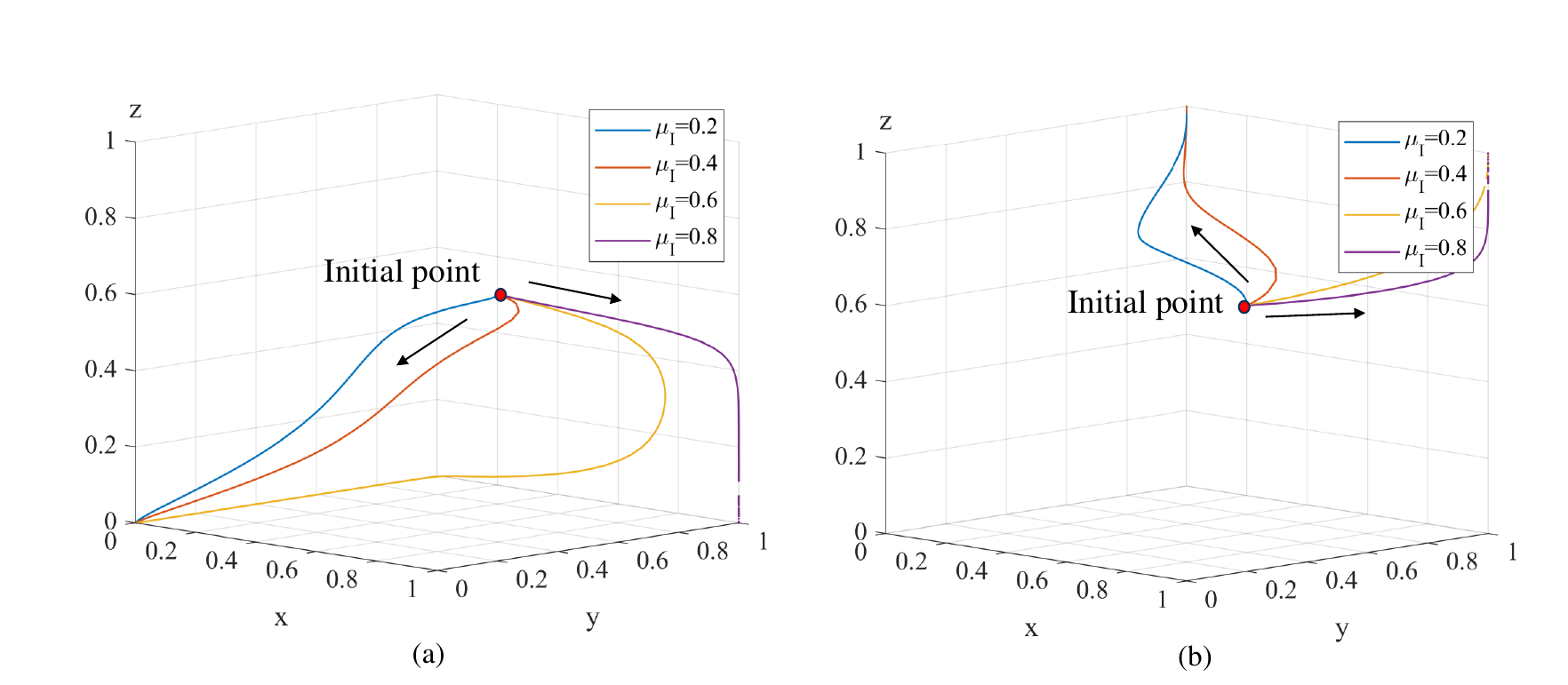}
\caption{$u_{I}$ evolutionary game simulation of array 1 and array 2}
\label{Fig17Fig18}
\end{figure}
    
It can be seen from Figure \ref{Fig17Fig18}a: under array 1, the critical value of $u_{I}$   is between 0.6 and 0.8. When $u_{I}$  is smaller than this critical value, the shipper, carrier, and platform finally select non-acceptance, non-acceptance, and no subsidy respectively, the dynamic system evolutionary stable point is (0,0,0); when $u_{I}$ is larger than this critical value, their choice become acceptance, acceptance and no subsidy respectively, so we have (1,1,0). According to the change rate of the evolution curve, it can also be seen that when the value of $u_{I}$ decreases from 0.6, the rate at which the dynamic system evolves to (0,0,0) speeds up. The lower the service level the shipper enjoys, the faster the shipper selects non-acceptance. It can be seen from Figure \ref{Fig17Fig18}b: under the condition of array 2, the critical value of $u_{I}$  is between 0.4 and 0.6. When $u_{I}$ is smaller than this critical value, the shipper, carrier, and platform finally select non-acceptance, acceptance, and subsidy respectively, and then the point is (0,1,1); when $u_{I}$ is larger than this critical value, the choice becomes acceptance, acceptance, and subsidy, with the point being (1,1,1). According to the change rate of the evolution curve, it can also be seen that when the value of $u_{I}$   is 0.6 instead of 0.8, the dynamic system evolves to (1,1,1) at a faster rate, therefore, appropriately improving the shipper service level is conducive to facilitate shipper to select acceptance.

\subsubsection{Carrier service level $u_{p}$   (referring to the service level enjoyed by the carrier)}
In order to study the influence of changing carrier service level $u_{p}$ on the system evolution results under $E_{7}$ and $E_{8}$ , the other parameters of array 1 and array 2 are kept unchanged, and the value of $u_{p}$  is changed by 0.2:0.2:0.8.The simulation results are shown in Figure \ref{Fig19Fig20}:

\begin{figure}[!htbp]
\centering
\includegraphics[width=.70\textwidth]{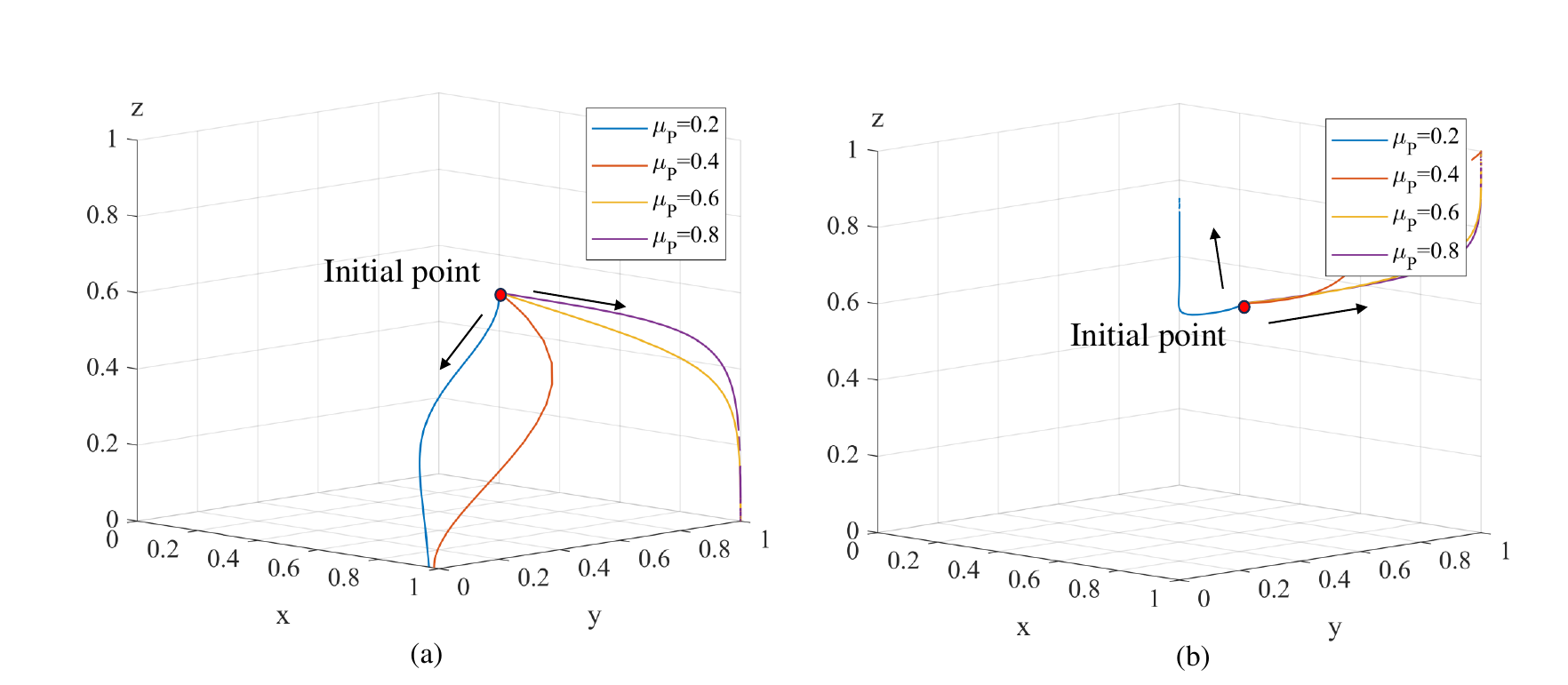}
\caption{$u_{P}$ evolutionary game simulation of array 1 and array 2}
\label{Fig19Fig20}
\end{figure}

Figure \ref{Fig19Fig20}a shows that for array 1, the critical value of is between 0.4 and 0.6. If  
$u_{p}$ is lower than this critical value, the shipper, carrier, and platform choose acceptance, non-acceptance, and no subsidy, respectively, resulting in a dynamic system evolutionary stable point of (1,0,0). If $u_{p}$ is higher than this critical value, they choose acceptance, acceptance, and no subsidy, respectively, resulting in a dynamic system evolutionary stable point of (1,1,0). As the value of $u_{p}$ decreases from 0.4, the dynamic system evolves to (1,0,0) faster. A lower carrier service level results in a faster shipper choice of non-acceptance. Similarly, Figure \ref{Fig19Fig20}b shows that for array 2, the critical value of $u_{p}$ is between 0.2 and 0.4. If $u_{p}$ is lower than this critical value, shipper, carrier, and platform choose acceptance, non-acceptance, and subsidy, respectively, resulting in a dynamic system evolutionary stable point of (1,0,1). If $u_{p}$ is higher than this critical value, they choose acceptance, acceptance, and subsidy, respectively, resulting in a dynamic system evolutionary stable point of (1,1,1). As the value of $u_{p}$ increases from 0.4, the dynamic system evolves to (1,1,1) slower. Hence, improving the carrier service level reasonably can help facilitate the shipper to select acceptance.

\subsubsection{The platform's fairness consideration intensity $\eta$ }
In order to study the influence of changing platform's fairness consideration intensity $\eta $ on the system evolution results under $E_{7}$  and $E_{8}$ , the other parameters of array 1 and array 2 are kept unchanged, and the value of    $\eta $ is changed by 0.2:0.2:0.8. The simulation results are shown in Figure \ref{Fig21Fig22}:

\begin{figure}[!htbp]
\centering
\includegraphics[width=.70\textwidth]{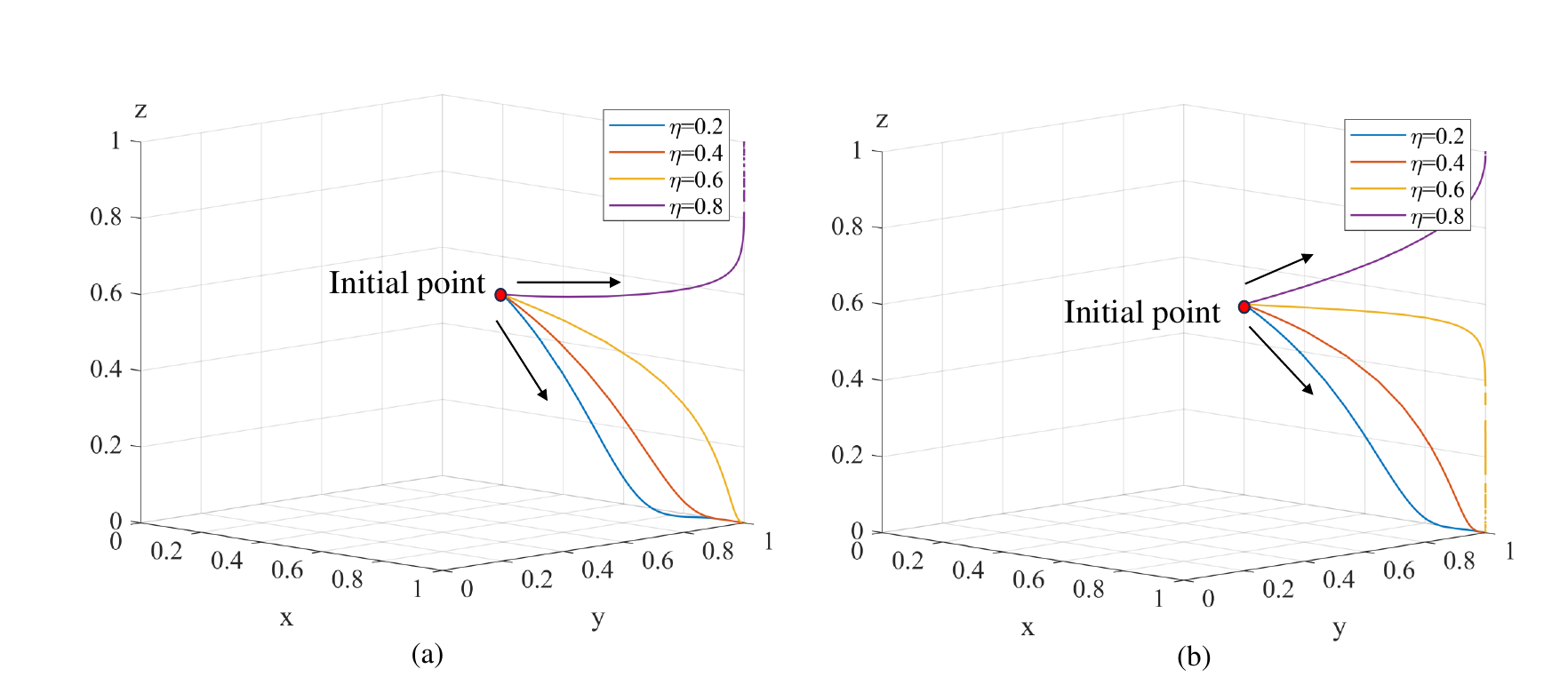}
\caption{$\eta$ evolutionary game simulation of array 1 and array 2}
\label{Fig21Fig22}
\end{figure}

Figure \ref{Fig21Fig22}a shows that for array 1, the critical value of is between 0.6 and 0.8. If   $\eta $ is lower than this critical value, the shipper, carrier, and platform choose acceptance, acceptance, and no subsidy, respectively, resulting in a dynamic system evolutionary stable point of (1,1,0). If $\eta $   is higher than this critical value, they choose acceptance, acceptance, and subsidy, respectively, so the point is (1,1,1). As $\eta $  decreases from 0.6, the dynamic system evolves to (1,1,0) faster, indicating that the platform tends to select no subsidy more and more. Based on the income matrix in Section \ref{6.1}, the system's evolution pattern under the platform's fairness consideration intensity is further analyzed. If the platform's fairness consideration intensity is relatively low, the reputation revenue brought by the subsidy strategy cannot make up for the management cost caused by the subsidy strategy, and the platform increasingly tends to select non-subsidy. This causes the shipper and the carrier to lose the subsidy income and also bear the platform's order commission costs. However, if the platform's fairness consideration intensity is higher than the critical value, the reputation revenue brought by the subsidy strategy can bring positive revenue growth even though management costs exist. The shipper and carrier can additionally obtain subsidy income, creating a win-win situation for the entire vehicle-cargo matching system. Thus, properly considering fairness can make the dynamic system reach a stable state of (1,1,1) and improve the three matching parties' overall income, which is more conducive to the online vehicle-cargo matching system's continuous and long-term operation. The influence of array 2 on system evolution results is the same as array 1's, as shown in Figure \ref{Fig21Fig22}b.

\section{Conclusions and future research}\label{6.0}
This paper proposes a tripartite vehicle-cargo matching evolutionary game framework consisting of a shipper, carrier, and platform in the online freight stowage industry. The framework explores the effect of matching fairness concerns on the online vehicle-cargo matching operation system using intuitionistic fuzzy sets theory. The parameter conditional restriction of expectant evolutionary stable equilibrium state under the cost-sharing mechanism is also analyzed. 

To reflect the unequal perception of bilateral vehicle-cargo matching subjects, the paper models the satisfaction realization degree of unilateral matching subjects using the intuitionistic fuzzy sets' membership, non-membership, and uncertainty information. Matching fairness concern is integrated into this model to obtain a satisfactory vehicle-cargo matching scheme through an adaptive interactive algorithm. The paper uses a tripartite evolutionary game-theoretic setting to examine the implications of fairness concerns on the online vehicle-cargo matching operation system. Furthermore, a cost-sharing contract is introduced to explore its impact on the evolutionary stable equilibrium state, making the research scenario setting more realistic.

\subsection{Main Conclusion}
This study yields several noteworthy findings:

The consideration of fairness in vehicle-cargo matching influences the resulting scheme, enabling decision-makers to create more flexible and coordinated strategies. The application of intuitionistic fuzzy set theory to transform the vehicle-cargo matching model outperforms traditional methods in balancing matching satisfaction and fairness. This is achieved by incorporating membership, non-membership, and uncertainty information pertaining to the satisfaction realization degree of unilateral matching subjects.

The formulation and implementation of vehicle-cargo matching schemes that address fairness concerns significantly impact the evolution of online vehicle-cargo matching operations. Tripartite evolutionary game simulation results indicate that the relative strength of various influencing factors, including fairness concerns, leads to distinct evolutionary path outcomes.

In order to achieve the desired system's evolutionary stable equilibrium state, certain conditions must be met. According to Lyapunov's first law, if a carrier accepts the recommended vehicle-cargo matching scheme while a shipper chooses non-acceptance, the shipper must assume the entire waiting response time cost, and vice versa.

\subsection{Managerial Insights}\label{8.0}
Addressing fairness concerns in vehicle-cargo matching leads to service innovation in online vehicle-cargo matching mechanisms. By modeling the unequal perceptions of bilateral vehicle-cargo matching subjects, the platform influences stakeholders' strategic selection. To meet decision-makers' expectations, platforms can develop satisfactory vehicle-cargo matching schemes with fairness concerns utilizing intuitionistic fuzzy logic.

Platform subsidy intensity, service level, and matching fairness concern intensity each have distinct effects on the evolutionary stable equilibrium state of the online vehicle-cargo matching system. To better serve stakeholders in the formulation and implementation of vehicle-cargo matching schemes, this study recommends strict control over platform subsidy intensity, reasonable adjustments to the platform's service levels for shippers and carriers, and proper consideration of matching fairness concern intensity. These measures will facilitate the evolution of the vehicle-cargo matching system toward the desired operational evolutionary stable equilibrium state.

The cost-sharing contract, which exists in the online freight stowage operational reality, influences the tripartite evolution of the online vehicle-cargo matching system. Decision-makers should determine the effective range of cost sharing to achieve the purpose of online vehicle-cargo matching system operation supervision.

\subsection{Limitations and future research}
Market fluctuation leads to imbalances in supply and demand within online vehicle-cargo matching transactions, making it intriguing to investigate their impact on decision-making within vehicle-cargo matching. Due to the influence of cognition and experience, shippers and carriers exhibit characteristics of “bounded rationality” when making matching decisions. As such, it would be beneficial to enhance the bilateral vehicle-cargo matching model by incorporating fairness concerns. Future research directions may involve exploring the specific details of interactive algorithms, which can serve as supplementary areas of study.



\begin{appendices}
\section{ }\label{Appendix A.} 


\setcounter{equation}{0}
\renewcommand\theequation{A.\arabic{equation}}
The scoring functions used in Section \ref{3.2}.
 \begin{equation}\label{A.1} 
S\left( A \right)=  \frac{\exp\left \{ u_{A}(x)-v_{A}(x)+H_{I}(x)\pi _{A}(x)  \right \} }{1+\pi_{A}^{2}(x) } 
\end{equation}                                                    
\noindent where:
\begin{equation}\label{shipperdeliverydate}
\pi_{A}\left ( x \right ) =1-u_{A}(x)-v_{A}(x)
\end{equation}
\begin{equation}\label{shipperdeliverydate}
H_{I}(x)=H(x)I(x)
\end{equation}
\begin{equation}\label{shipperdeliverydate}
H(x)=\frac{H(u_{A}(x),v_{A}(x))+H(v_{A}(x),u_{A}(x))}{2}
\end{equation}
\begin{equation}\label{shipperdeliverydate}
H(u_{A}(x),
v_{A}(x))=u_{A}(x)\log_{2}\frac{u_{A}(x)}{(u_{A}(x)+v_{A}(x))/2} +(1-u_{A}(x))\log_{2}\frac{1-u_{A}(x)}{1-(u_{A}(x)+v_{A}(x))/2}
\end{equation}
\begin{equation}\label{shipperdeliverydate}
I(x)=
\begin{cases}
1  \quad\quad u_{A}(x)>v_{A}(x)\\
0  \quad\quad u_{A}(x)=v_{A}(x) \\
-1 \quad u_{A}(x)<v_{A}(x)
\end{cases}
\end{equation}
                                           
 \noindent$u_{A}(x)$,$v_{A}(x)$,$\pi_{A}(x)$ respectively represent the degree of membership, non-membership, and hesitation; $H(x)$ is the intuitionistic fuzzy cross-entropy, which is used to represent the interaction between the degrees of membership and non-membership.


 

\section{ } \label{Appendix B.} 

\setcounter{equation}{0}
\renewcommand\theequation{B.\arabic{equation}}

The optimization model determines the reliability indicator evaluation criteria weight in Section \ref{3.2}.

\begin{equation}\label{3.12}
\max V(w)={\textstyle \sum_{i=1}^{p}}  {\textstyle \sum_{j=1}^{q}} v_{ij} w_{j} 
\end{equation}
\begin{equation}\label{3.13}
\hbox{s.t.}~ {\textstyle \sum_{j=1}^{q}}w_{j}^{2}=1,w_{j}\ge0,j=1,2,\dots,q
\end{equation}
\noindent where $v_{ij} $ is the criterion evaluation value and $w_{j}$ is the criterion weight. For the above nonlinear optimization model, the Lagrange method is used to get the optimal criterion weight:
\begin{equation}\label{weightdetermine1}
w_{j}=\frac{1}{ {\textstyle \sum_{j=1}^{q}} {\textstyle \sum_{i=1}^{p}v_{ij}}}  {\textstyle \sum_{i=1}^{p}}v_{ij},j=1,2,\cdots,q 
\end{equation}

\section{ } \label{Appendix C.}

\setcounter{equation}{0}
\renewcommand\theequation{C.\arabic{equation}}
The Relative Superiority Degree of Adjacent Targets (RSDAT) method in Section \ref{3.4}.

First, the indicators' importance is sorted and recorded in descending order as $o_{1} \succ o_{2}\cdots \succ o_{k}$ , then, the binary comparison of indicators' importance is made,  $I_{ab}$ shows the relative importance scale value, see Table \ref{degreetable}:

\renewcommand\thetable{C.\arabic{table}}    
\setcounter{table}{0}

\begin{table}[!htp]
\begin{center}
\caption{Determination of scale value}\label{degreetable}
\begin{tabular}{cc}
\hline
Degree of importance                   & Scale value                        \\ \hline
$O_{a}$ is more important than $O_{b}$ & 0.5\textless{}$I_{ab}$\textless{}1 \\
$O_{b}$ is more important than $O_{a}$ & 0\textless{}$I_{ab}$\textless{}0.5 \\
$O_{a}$ is as important as $O_{b}$     & $I_{ab}$=0.5                       \\ \hline
\end{tabular}
\end{center}
\end{table}

 In addition,$I_{ab}=1-I_{ab}$, $I_{a,a+1}(a=1,2,\cdots,k-1)$ shows the relative importance scale value of adjacent targets, given by the decision maker. Based on the above definitions, rules and recursion principles, the relative importance scale value of any two indicators can be obtained from $I_{a,a+1}$, namely:
\begin{equation}\label{weightdetermine1}
I_{ab}=I_{a,b-1}+2(I-I_{a,b-1})(I_{b-1,b}-0.5)
\end{equation}  

$I$ shows the indicators' ordered binary comparison matrix according to complementary relationships:

\begin{equation}\label{matrix}
I=\left [ I_{ab} \right ]_{k\times k} =\begin{bmatrix}
  I_{11}& I_{12} & \cdots  &I_{1k} \\
  I_{21}& I_{22} & \cdots  &I_{2k}  \\
  \vdots & \vdots  & \vdots   &\vdots  \\
  I_{k1}& I_{k2} & \cdots  &I_{kk} 
\end{bmatrix}
\end{equation}
                                                                     
\noindent$w_{a}^{'}$ represents the relative importance of corresponding indicator, which is equal to the sum of each row in the matrix $I$ excluding its own comparison:
\begin{equation}\label{weightdetermine1}
w_{a}^{'}={\textstyle \sum_{b=1}^{k}} I_{ab}
  \qquad  a,b=1,2,\cdots,k,a\ne b 
\end{equation}
                                                          
Normalizing $w_{a}^{'}$ 
\begin{equation}
w_{a}=\frac{w_{a}^{'}}{ {\textstyle \sum_{a=1}^{}} w_{a}^{'}}  \quad  a=1,2\cdots,k  \end{equation}

\end{appendices}

\end{document}